\documentclass[a4paper,10pt]{article}

\usepackage[dvipdfm,%
 bookmarks=true,%
 bookmarksnumbered=true,%
 colorlinks=true,%
 pdftitle={},%
 pdfauthor={Yokoyama},%
 pdfsubject={},%
 pdfkeywords={}]{hyperref}

\makeatletter
\def\blfootnote{\gdef\@thefnmark{}\@footnotetext}
\makeatother

\usepackage{soul}

\usepackage{amsthm}
\usepackage{latexsym}
\usepackage{amsmath}
\usepackage{amscd}
\usepackage{amssymb}
\usepackage{authblk}
\DeclareSymbolFont{bbold}{U}{bbold}{m}{n}
\DeclareMathSymbol{\bbomega}{\mathord}{bbold}{"7F}

\usepackage{url}
\usepackage[all]{xy}

\usepackage{color}

\newcommand\RCAo{\mathsf{RCA_0}}

\newcommand\ACAo{\mathsf{ACA_0}}

\newcommand\PCAo{\Pi^1_1\text{-}\mathsf{CA_0}}

\newcommand\A{\forall}
\newcommand\N{\mathbb{N}}

\newcommand\Q{\mathbb{Q}}
\newcommand\R{\mathbb{R}}

\newcommand\lh{\mathrm{lh}}

\newcommand\rest{{\upharpoonright}}

\newcommand{\mr}[1]{\mathrm{#1}}
\newcommand{\mc}[1]{\mathcal{#1}}

\renewcommand{\labelenumi}{$\arabic{enumi}.$}

\newcommand\fin{\mathrm{fin}}

\newcommand\UF{\mathrm{UF}}
\newcommand\MF{\mathrm{MF}}
\newcommand\NP{\mathrm{NP}}
\newcommand\upcl{\mathrm{upcl}}

\theoremstyle{plain}
\newtheorem{thm}{Theorem}[section]

\newtheorem*{theorem*}{Theorem}

\newtheorem*{claim*}{Claim}
\newtheorem{prop}[thm]{Proposition}
\newtheorem{lem}[thm]{Lemma}
\newtheorem{cor}[thm]{Corollary}

\theoremstyle{definition}
\newtheorem{defi}{Definition}[section]
\newtheorem{definition}[defi]{Definition}
\newtheorem{rem}[thm]{Remark}

\newtheorem{question}[defi]{Question}

\begin{document}
\newpage

\title{Quasi-Polish spaces and spaces of filters in second-order arithmetic}

\author[1]{Yuzuki Kaneko}
\author[2]{Keita Yokoyama}

\affil[1,2]{Mathematical Institute, Tohoku University}
\affil[1]{\sf{kaneko.yuzuki.q4@dc.tohoku.ac.jp}}
\affil[2]{\sf{keita.yokoyama.c2@tohoku.ac.jp}}

\date{}


\maketitle

\begin{abstract}
The class of quasi-Polish spaces admits several equivalent representations, including UF spaces, NP spaces, ideal spaces, $\mathbf{\Pi}_2^0$ subspaces of $\mathcal{P}(\mathbb{N})$, and sober spaces of countably presented frames. In this paper, we formalize these structures within second-order arithmetic and conduct a systematic reverse mathematical analysis of the transitions between them.
\end{abstract}

\tableofcontents

\section{Introduction}

Polish spaces are central objects of study in classical descriptive set theory, characterized by several fundamental properties: they are exactly the $\mathbf{\Pi}_2^0$ subspaces of the Hilbert cube, they admit continuous surjections from the Baire space $\mathbb{N}^{\mathbb{N}}$, and they can be constructed as completions of countable dense sets \cite{kechris2012classical}. 
In the framework of second-order arithmetic, which forms the standard setting for reverse mathematics \cite{SOSOA}, a Polish space is traditionally formalized as a pair $(A, d)$ consisting of a countable set $A$ and a pseudo-metric $d$, where points are represented as effective Cauchy sequences. This formalization has been extensively utilized to examine the logical strength of theorems in analysis and topology; for instance, Fern\'andez-Duque, Shafer, and the second author analyzed Ekeland's variational principle within this framework \cite{fernandez2020ekeland, fernandez2023caristi}.

While the representation of points as equivalence classes of Cauchy sequences is a standard approach for separable metric spaces, it can be generalized to second-countable $T_0$ spaces by replacing the metric with a quasi-metric. In this context, the class of quasi-Polish spaces, whose structural properties were extensively investigated by de Brecht \cite{deBr2013}, serves as a common generalization of $\omega$-continuous domains and Polish spaces. Notably, quasi-Polish spaces possess characterizing properties analogous to those of Polish spaces; for instance, they are precisely the $\mathbf{\Pi}_2^0$ subspaces of the powerset space $\mathcal{P}(\mathbb{N})$ (identified with the countable power of the Sierpi\'{n}ski space) endowed with the Scott topology.

Another approach to formalizing non-metrizable spaces involves the use of poset spaces (where ``poset'' stands for preordered set). Mummert/Stephan \cite{mummert2010topological} investigated the properties of two types of poset spaces: UF spaces and MF spaces (see also \cite{mummert2006reverseMF, mummert2005reverse}).
Regarding the former, Honda \cite{honda-unpublished} observed that the class of UF spaces is equivalent to that of quasi-Polish spaces. 
Ideal spaces, investigated from the perspective of computable topology in \cite{de2020overt, de2020some}, are defined similarly to poset spaces. In particular, their definition can easily be translated into that of NP spaces, which constitute another subclass of poset spaces. The class of $\mathbf{\Pi}_2^0$ subspaces of $\mc P(\N)$ is known to be naturally equivalent to that of ideal spaces \cite{BKS2024}. 
Furthermore, from the perspective of point-free topology, these spaces are known to be equivalent to sober spaces whose frames of open sets are countably presented as shown by Heckmann \cite{heckmann2015spatiality}. Regarding their relationship with locales, de Brecht \cite{de2019note} demonstrated that the functor from countably based sober spaces to locales, when restricted to the subcategory of quasi-Polish spaces, preserves certain properties, such as finite products.
Other related lines of research on second-countable spaces includes the reverse mathematical analysis of second-countable Noetherian spaces via well-quasi-orders by Frittaion/Hendtlass/Marcone/Shafer/Van der Meeren \cite{frittaion2016noetherian},
studies of countable and second-countable spaces initiated by Dorais \cite{dorais2011compactCountable} and investigated by several groups (e.g., \cite{benham2024ginsburgSands, genovesi2026regularCSCS}), and second-countable spaces from the viewpoint of higher-order reverse mathematics by Sanders \cite{sanders2025secondCountable}.

Combining the results established above, we obtain the following equivalences.

\begin{prop}\label{prop:equivalent-representations}
The following representations of second-countable topological spaces are mutually equivalent:
\begin{enumerate}
  \item Quasi-Polish spaces;
  \item UF spaces;
  \item NP spaces;
  \item Ideal spaces;
  \item $\mathbf{\Pi}_2^0$ subspaces of $\mathcal{P}(\mathbb{N})$;
  \item Sober spaces associated with countably presented frames.
\end{enumerate}
\end{prop}

Proposition \ref{prop:equivalent-representations} naturally answers several open questions in \cite{mummert2005reverse}.


While these representations are mathematically equivalent in classical set theory, one often moves between them depending on the application, such as in descriptive set-theoretic, computational, or point-free topological contexts. 
For instance, analytic results in Polish spaces often have direct analogues in quasi-Polish spaces, such as the generalized Hurewicz theorem \cite{de_Brecht_2018}, where the quasi-metric framework facilitates a direct comparison with the classical Polish setting.
In this study, we investigate the methods for switching between these representations within the framework of second-order arithmetic. Such translations are crucial, as they allow for the ``cross-pollination'' of results between diverse formalizations; for example, importing metrization theorems developed for poset spaces \cite{mummert2005reverse} into the theory of $\mathbf{\Pi}_2^0$ subspaces of $\mathcal{P}(\mathbb{N})$. However, as we demonstrate, the ability to perform these transitions is not always trivial and often requires significant axiomatic strength beyond $\mathsf{RCA}_0$.

In this paper, we introduce several formalizations of quasi-Polish spaces within second-order arithmetic and conduct a systematic reverse mathematical analysis of the transitions between these different representations. Our main results are summarized in the following theorems.

\begin{thm}\label{intro1}
The following implications hold over $\RCAo$:
\begin{enumerate}
\item Quasi-Polish space $\implies$ UF space.
\item UF space $\implies$ $\mathbf{\Pi}_2^0$-subspace of $\mc P(\N)$.
\item NP space $\iff$ ideal space.
\item NP space $\iff$ $\mathbf{\Pi}_2^0$-subspace of $\mc P(\N)$.
\item $\mathbf{\Pi}_2^0$-subspace of $\mc P(\N)$ $\iff$ sober space of countably presented frame.
\end{enumerate}
\end{thm}

\begin{thm}\label{intro2}
Over $\RCAo$, the following statements are equivalent:
\begin{enumerate}
\item $\PCAo$.
\item $\mathbf{\Pi}_2^0$-subspace of $\mc P(\N)$ $\implies$ UF space.
\item UF space $\implies$ quasi-Polish space.
\item $\mathbf{\Pi}_2^0$-subspace of $\mc P(\N)$ $\implies$ quasi-Polish space.
\end{enumerate}
\end{thm}

Additionally, we introduce three distinct formalizations of countably presented frames. We show that two of these representations, together with NP spaces, are mutually equivalent within $\mathsf{RCA}_0$, whereas the third is equivalent to the others only under the assumption of $\mathsf{ACA}_0$.

The remainder of this paper is organized as follows. In Section 2, we review the standard definitions and properties of quasi-Polish spaces, poset spaces, and countably presented frames. In Section 3, we discuss their formalization in second-order arithmetic and examine the axiomatic strength required to prove their equivalences.




\section{Poset spaces and quasi-Polish spaces}
This section provides the definitions (in usual mathematics) and establishes the equivalence between different representations, alongside the basic properties of the class of quasi-Polish spaces.
Although several of these definitions and results are already established in \cite{mummert2005reverse,mummert2010topological, deBr2013}, we recall the constructions explicitly because their proofs determine the formalizations and axiomatic bounds analyzed in Section 3.

\subsection{Poset spaces and ideal spaces}

We first review the definitions of poset spaces studied in \cite{mummert2005reverse,mummert2010topological}. For a preordered set $(P, \leq)$ and elements $p, q \in P$, we write $p < q$ to mean $p \leq q$ and $p \neq q$. 
\begin{defi}[poset spaces]
Let $(P,\leq)$ be a pair of a set $P$ and a binary relation ${\leq}\subseteq P\times P$.
$(P,\leq)$ is a \textit{poset} if $\leq$ is the preorder on $P$.
A subset $F$ of a poset $P$ is called a \textit{filter} on $P$ if it satisfies the following conditions.
\begin{enumerate}
\renewcommand{\labelenumi}{(\roman{enumi})}
 \item$\forall p,p^{\prime}\in P((p\in F\land p\leq p^{\prime})\rightarrow p^{\prime}\in F)$ (upward closed);
 \item
$\forall p,p^{\prime}\in P((p,p^{\prime}\in F)\rightarrow\exists r\in F(r\leq p,p^{\prime}))$;
 \item
$F\neq\emptyset$.
\end{enumerate}
A space of filters on $P$ is denoted by $\mathbb{F}(P)$.
\begin{enumerate}
 \item A filter $F$ is \textit{unbounded} if $\lnot\exists r\in P(\forall p\in F(r<p))$. The UF space of $P$, denoted by $\UF(P)$ is the set of unbounded filters on $P$.
 \item A filter $F$ is \textit{maximal} if there is no filter which has $F$ as a proper subset. For instance, if a filter $F$ satisfies $\forall r\in P\,((r\notin F)\rightarrow\exists p\in F\,(r\bot p))$, then $F$ is clearly a maximal filter.
 Here, two elements $r,p\in P$ are incompatible, denoted by $r\bot p$, if there is no $q\in P$ such that $q\leq p$ and $q\leq r$. The MF space of $P$, denoted by $\MF(P)$ is the set of maximal filters on $P$.
 \item A filter $F$ is \textit{non-principal} if $\forall p\in F\,\exists p^{\prime}\in F(p^{\prime}<p\land p\nless p^{\prime})$. A space of non-principal filters is denoted by $\NP(P)$.
\end{enumerate}
A space of filters is topologized by a basis $\{N_{p}:p\in P\}$, where $N_{p}=\{F\in \mathbb{F}(P):  p\in F\}$. The topologies on $\UF(P),\MF(P)$ and $\NP(P)$ are relative topologies of $\mathbb{F}(P)$.
We collectively refer to $\mathbb{F}(P)$, $\UF(P)$, $\MF(P)$ and $\NP(P)$ as poset spaces.
\end{defi}
Note that assuming every filter $F$ is nonempty, which does not restrict the class of poset spaces. Indeed, given a poset $(P, \leq_P)$, let $P^{\prime} := P \cup \{1_P\}$ be the poset obtained by adjoining the maximal element $1_P$ (i.e., $p \leq_{P^{\prime}} 1_P$ for all $p \in P$). Then, the space of all (possibly empty) filters on $P$ is naturally homeomorphic to the space of nonempty filters on $P^{\prime}$, i.e., $\mathbb{F}(P) \cup \{\emptyset\} \cong \mathbb{F}(P^{\prime})$.

Poset spaces exist as analogues of ideal spaces; they represent essentially the same structures, differing only in the orientation of the order and whether the relation is assumed to be a preorder or a transitive relation.
\begin{defi}[Ideal spaces]
Let $\prec_P$ be a transitive relation on a subset $P \subseteq \N$. A subset $F \subseteq P$ is called an \textit{ideal} if it satisfies the following conditions:
\begin{enumerate}
\renewcommand{\labelenumi}{(\roman{enumi})}
 \item$\forall p,p^{\prime}\in P((p\in F\land p^{\prime}\prec_P p)\rightarrow p^{\prime}\in F)$;
 \item
$\forall p,p^{\prime}\in P((p,p^{\prime}\in F)\rightarrow\exists r\in F(p,p^{\prime}\prec_P r))$;
 \item
$F\neq\emptyset$.
\end{enumerate}
We abbreviate the statement that a set $F$ is an ideal on $(P,\prec_P)$ as $F\in\mathbf{I}(P,\prec_P)$. 
A code for the open basis of $\mathbf{I}(P,\prec_P)$ is given by $P$. A point (i.e., an ideal) $x\in\mathbf{I}(P,\prec_P)$ belongs to the basic open set $N_p$, coded by $p\in P$ if $p\prec_P q\lor p=q$ for some $q\in x$.
\end{defi}

Although we defined poset spaces using a preorder $P$ in the definition above, restricting $P$ to be a partial order does not alter the class of resulting spaces. Indeed, let $\sim_P \;:=\, \{(a,b) \in P \times P : a \leq b \,\land\, b \leq a\}$. Then $\sim_P$ is an equivalence relation, and the quotient $P/{\sim_P}$ forms a partial order such that $\mathbb{F}(P) \cong \mathbb{F}(P/{\sim_P})$. For ideal spaces, it was shown by \cite{BKS2024} that any transitive relation $\prec$ considered in \cite{de2020overt, de2020some} admits a strict partial order $\sqsubset$ such that $\mathbf{I}(\prec) \cong \mathbf{I}(\sqsubset)$.

We sometimes denote the intersection of $\NP(P)$ and $\UF(P)$ by $\NP\UF(P)$. A poset $P$ is said to be \textit{handy} if $\UF(P)=\NP(P)$.

Throughout this paper, we primarily work with countable posets. We remark that every second-countable poset space is \textit{countably based} (it can be formed from a countable poset) \cite{mummert2010topological}. 

\begin{prop}\label{1-typicalization}
Let $P$ be a countable poset. 
\begin{enumerate}
 \item There exists a countable handy poset $P^{\prime}$ such that $\UF(P)\cong\UF(P^{\prime})$.
 \item There exists a countable handy poset $P^{\prime}$ such that $\mathbb{F}(P)\cong\UF(P^{\prime})$.
\item There exists a countable handy poset $P^{\prime}$ such that $\UF(P^{\prime})$ is isomorphic to $\NP\UF(P)$, where $\NP\UF(P)=\UF(P)\cap\NP(P)$.
\end{enumerate}
\end{prop}
%
\begin{proof}
We fix an enumeration $P$ as $\langle p_i : i \in \mathbb{N} \rangle$. If $P$ is finite, we instead fix an enumeration $P$ as $\langle p_i : i < k \rangle$ for some $k \in \mathbb{N}$.
\begin{enumerate}
\item Let $P^{\prime}=\{(p_{i},n)\in P\times\N :\forall m<n(p_{m}\nless p_{i})\}$. The order on $P^{\prime}$ is defined by $(p,n)\leq(q,m)\Leftrightarrow (p\leq q\land n> m)\lor(p= q\land n= m)$. This poset $P^{\prime}$ is handy since every element $(p_{i}, n)$ satisfies $(p_{i}, n)>(p_{i}, n+1)\lor(p_{i}, n)>(p_{n}, n+1)$, and for any non-principal filter $F$, $\forall n\exists i\, (p_{i}, n)\in F$. We define an isomorphism $f$ from $\UF(P)$ to $\UF(P^{\prime})$ as follows:
\[
f(F)=\{(p,n)\in P^{\prime}:p\in F\}
\]
for $F\in \UF(P)$. It is clear that $f(F)$ is upward closed. Furthermore, if $(p_{i}, n),(p_{j}, m)\in f(F)$, there exists $p_{k}\in F(p_{i},p_{j}>p_{k}\land\,\forall l<\max\{m,n\}+1 (p_{l}\nless p_{k}))$ because $F$ is unbounded. Thus, $(p_{i}, n),(p_{j}, m)>(p_{k},\max\{m,n\}+1)\in f(F)$, showing that $f$ is well-defined. Clearly, $f$ is injective and continuous. 
 Moreover, since $((p_{i},n)\in G\land(p_{i},m)\in P^{\prime})\rightarrow(p_{i},m)\in G$ for every $G\in \UF(P^{\prime})$ as $\exists k((p_i,n)\leq(p_k,m+1)\in G)$ when $m < n$, $f$ is surjective. Finally, $f(N_{p_{i}})=\bigcup\{N_{(p_{i},n)}:n\in\N,(p_i,n)\in P^{\prime}\}$
is open in the relative topology of $\UF(P^{\prime})$.

\item Let $P^{\prime}=\{(p_{i},n)\in P\times\N :i\leq n\}$. The order on $P^{\prime}$ is defined by $(p,n)\leq(q,m)\Leftrightarrow (p\leq q\land n> m)\lor(p= q\land n= m)$. This poset $P^{\prime}$ is handy because a filter $F$ is unbounded if and only if $\forall n\in\N\,\exists (p,n)\in F$. We define an isomorphism $f$ from $\mathbb{F}(P)$ to $\UF(P^{\prime})$ as follows:
\[
f(F)=\{(p_i,n)\in P^{\prime}:p_i\in F\}
\]
where $F\in\mathbb{F}(P)$. Then $f$ is injective, and $f(F)$ is an unbounded filter. Furthermore, for any filter $G\in \UF(P^{\prime})$, the set $G^{\prime}:=\{p_{i}\in P:\exists n\in\N\,(p_{i},n)\in G\}=\{p_{i}\in P:(p_{i},i)\in G\}$ is a filter on $P$ such that $f(F)=G$, which shows that $f$ is surjective. Thus, there exists the inverse mapping $f^{-1}$ induced by the assignment $P^{\prime}\rightarrow P:(p_{i},i)\mapsto p_{i}$. It follows that $f$ is continuous. Moreover, since $((p,n)\in G\land n\leq m)\rightarrow(p,m)\in G$ holds for every $G\in \UF(P^{\prime})$, we have $f(N_{p_{i}})=\bigcap_{n\geq i}N_{(p_{i},n)}=N_{(p_i,i)}$. This set is open in the relative topology of $\UF(P^{\prime})$, which implies that $f$ is an open mapping.

\item Define a poset $P'' \subseteq\N\times P$ by $P^{\prime\prime}=\{(n,p)\in\N\times P:\forall j\leq n(p\not\geq_P p_j)\}$ equipped with the order $(n,p)\geq_{P^{\prime\prime}}(m,q)\Leftrightarrow (p>_Pq\land m>n)\lor(p=q \land n=m)$. Then $P''$ satisfies $\NP\UF(P)\cong\NP(P'')$ and $\NP(P'')\subseteq\UF(P'')$. (see Proposition~\ref{prop:1-NPtoNPUF} for a detailed proof). For any poset $R$, let $\hat{R}=\{p\in R: \exists q\in R\,(p>_Rq\land p\nless_R q)\}$ be a subposet of a poset $R$. We define a sequence of subposets by transfinite induction: $P_0=P''$, $P_{\beta+1}=\hat{P_{\beta}}$ and $P_{\gamma}=\bigcap_{\beta<\gamma} P_{\beta}$ for limit ordinals $\gamma$. Since $P$ is countable and the sequence is non-increasing, there exists the least ordinal $\alpha<\omega_1$ such that $P_{\alpha}=P_{\alpha+1}$. The resulting poset $P^{\prime}=P_{\alpha}$ satisfies the required conditions. \qedhere
\end{enumerate}%
\end{proof}

Furthermore, the following statements are provable.

\begin{prop}[NP spaces]\label{prop:1-NPtoNPUF}
Let $P$ be a countable poset. Let $(Q,\prec_Q)$ be a pair consisting of a countable set $Q$ and a transitive relation $\prec_Q$ on $Q$.
\begin{enumerate}
    \item There exists a countable poset $P^{\prime}$ such that $\NP(P)\cong\NP\UF(P^{\prime})$.
    \item Conversely, there exists a countable poset $P^{\prime\prime}$ such that $\NP\UF(P)\cong\NP(P^{\prime\prime})$. 
\item There exists a countable poset $P^{\prime}$ such that $\mathbf{I}(Q,\prec_Q)\cong\NP(P^{\prime})$.
\end{enumerate}
\end{prop}
\begin{proof}
We fix an enumeration $P$ as $\langle p_i : i \in \mathbb{N} \rangle$. If $P$ is finite, we instead fix an enumeration $P$ as $\langle p_i : i < k \rangle$ for some $k \in \mathbb{N}$.
\begin{enumerate}
\item Let $P^{\prime}=P$ as a set, equipped with the order $\leq_{P^{\prime}}$ defined by $p_i\geq_{P^{\prime}}p_j\Leftrightarrow (p_i>_Pp_j\land i<j)\lor (i=j)$, where $\leq_{P}$ denotes the order on $P$. We define a bijection $f:\NP(P)\to\NP\UF(P)$ by $f(F)=F$. Under this construction, for every non-principal filter $F\in\NP(P)$, its image $f(F)$ belongs to $\NP(P')$ because for every $p_i\in F$, there exists $p_j\in F$ such that $(p_i>_Pp_j\land i<j)$. Furthermore, since the order requires the index to increase as elements strictly decrease, every non-principal filter in $P^{\prime}$ is necessarily unbounded, thus belonging to $\UF(P^{\prime})$. Conversely, for any $F'\in\NP\UF(P')$, its preimage $f^{-1}(F')=F'$ belongs to $\NP(P)$, since for each $i\in\N$, the set of indices $j\in\N$ satisfying $(p_i>_Pp_j\land i\geq j)$ is  finite (at most $i$). Moreover, the equality $f(\{F\in\NP(P):p\in F\})=\{F\in\NP(P'):p\in F\}$ shows $f$ is naturally open and continuous.
\item We define the poset $P^{\prime\prime}$ by:
\[
P^{\prime\prime}=\{(n,p)\in\N\times P:\forall j\leq n(p\not\geq_P p_j)\},
\] equipped with the order $(n,p)\geq (m,q)\Leftrightarrow (p>_Pq\land m>n)\lor(p=q \land n=m)$. For any $F\in\NP(P)\cap\UF(P)$, the mapping $f(F)=\{(n,p)\in P^{\prime\prime}:p\in F\}$ yields an element of $\NP(P^{\prime\prime})$. Indeed, for all $n\in\N$, there exists $p\in F$ such that $\forall j\leq n(p\not\geq_P p_j)$. Conversely, for every $G\in\NP(P^{\prime\prime})$, the set $g(G)=\{p\in P:(0,p)\in G\}$ belongs to $\NP(P)\cap\UF(P)$. Indeed, if $G\in\NP(P^{\prime\prime})$, then for all $n\in\N$, there exist elements $(n,p),(m,q)\in G$ such that $(n,p)>(m,q)\in G$. This $p$ satisfies $\forall j\leq n(p\not\geq_P p_j)$ and $(0,p)>(m,q)\in G$, which shows that $g(G)$ is unbounded. Both $f$ and $g$ are continuous, and each is the inverse of the other.
\item Let us define a partial order $\leq$ on $P^{\prime}:=Q\times \N$ by setting $(a,m) \leq (b,n) \leftrightarrow (n < m \land b \prec_Q a)\lor(n=m\land a=b)$.
Then $f(F)=\{(a,m)\in Q\times\N:a\in F\}$ provides an open continuous bijection $f:\mathbf{I}(Q,\prec_Q)\to\NP(P^{\prime})$. \qedhere
\end{enumerate}%
\end{proof}

By connecting these results, we may assume without loss of generality that all these poset spaces with a countable poset ($\mathbb{F}(P)$, $\UF(P)$, $\NP(P)$, and $\NP\UF(P)$) and ideal spaces can be represented as $\UF(Q)$ spaces with some handy poset $Q$.
 
\subsection{Quasi-Polish spaces}
Next, we review the definition of quasi-Polish spaces and gather some of their basic properties.

\begin{defi}[quasi-Polish spaces]

A \textit{quasi-metric} on a set $X$ is a function $d:X\times X\rightarrow \mathbb{R}$ that satisfies the following.
\begin{enumerate}
\renewcommand{\labelenumi}{(\roman{enumi})}
\item$\forall x,y\in X(d(x,y)\geq0)$;
\item$\forall x,y\in X(d(x,y)=d(y,x)=0\rightarrow x=y)$;
\item$\forall x,y,z\in X(d(x,y)+d(y,z)\geq d(x,z))$.
\end{enumerate}
A quasi-metric $d$ generates a topology on $X$ with a basis $\mathcal{B}=\{B_{x,r} :x\in X, r\in\mathbb{R}_{>0}\}$, where $B_{x,r}:=\{y\in X:d(x,y)<r\}$.
A sequence $(a_{n})_{n}$ in a quasi-metric space $X$ is called \textit{left-Cauchy} if:
\[
\forall\varepsilon >0\,\exists N\in\mathbb{N}\,\forall m>n>N(d(a_{n},a_{m})<\varepsilon)
\]
We say that $(X,d)$ is \textit{Smyth-complete} if every left-Cauchy sequence in $X$ converges to a point $a\in X$ in the following sense: $\forall\varepsilon >0\,\exists N\in\mathbb{N}\,\forall n>N(d(a,a_{n})<\varepsilon\land d(a_{n},a)<\varepsilon)$. 

A quasi-metric $d$ induces a metric $\hat{d}$ defined by $\hat{d}(x,y)=\max\{d(x,y),d(y,x)\}$. We say that $(X,d)$ is \textit{bicomplete} if every $\hat{d}$-Cauchy sequence in $X$ converges to a point $a\in X$. Every Smyth-complete quasi-metric space is bicomplete.

A \textit{quasi-Polish space} is a second-countable and Smyth-completely quasi-metrizable space. For comparison, recall that a \textit{Polish space} is a separable and completely metrizable space.
\end{defi}
De Brecht \cite{deBr2013} extensively studied the properties of quasi-Polish spaces.

\begin{defi}[Universal space]
We can show that $\mc P(\N)$, equipped with an appropriate topology, is both a UF space and a Smyth-complete quasi-metric space as follows.

Let $\mc P_{\fin}(\N)$ be the poset consisting of all finite subsets of $\N$, ordered by reverse inclusion: 
\[
p\geq q\Leftrightarrow p\subseteq q
\]
for $p,q\in\mc P_{\fin}(\N)$.
Then the space $\mathbb{F}(\mc P_{\fin}(\N))\cup\{\emptyset\}$ can be identified with $\mc P(\N)$ by associating each element $F$ with the set $\{n \in \N :\exists p\in F(n\in p)\}$. Under this identification, the topology on $\mc P(\N)$ is generated by the basis: $\{N_p:p\in\mc P_{\fin}(\N)\}\cup\{\mc P(\N)\}$ where $N_p=\{F\in\mc P(\N):p\subseteq F\}$. 

With this notation, $\mc P(\N)$ is Smyth-completely quasi-metrizable with respect to the following quasi-metric:
\[
d(F,G)=2^{-j}
\]
 where $j$ is the least number such that $j\in F\land j\notin G$. 
\end{defi}
Every second-countable $T_0$ space is homeomorphic to a subspace of $\mc P(\N)$. Let $X$ be a second-countable $T_0$ space, where $\{ V_i:i\in\N\}$ is a countable basis of $X$. Then the map $\iota:X\rightarrow \mc P(\N)$ defined as $\iota(x)=\{i\in\N:x\in V_i\}$ is an embedding.

\begin{defi}[Borel sets, {\cite[Def.1]{deBr2013}}]
In a quasi-metric space, we can define open sets, closed sets, $F_{\sigma}$ sets, and $G_{\delta}$ sets, etc., in the usual manner.
However, unlike in metric spaces, an open set in a quasi-metric space is not necessarily an $F_{\sigma}$ set.
Similarly, a closed subset is not always a $G_{\delta}$ set in a quasi-metric space. Therefore, we define the Borel hierarchy of a topological space $(X,\tau)$ as follows:
\begin{enumerate}
\item$\mathbf{\Sigma}_{1}^{0}$ is the collection of open sets $\tau$.
\item$\mathbf{\Sigma}_{\alpha}^{0}=\{\bigcup_{i\in\omega}(B_{i}\backslash B^{\prime}_{i}):B_{i},B^{\prime}_{i}\in\bigcup_{\beta<\alpha}\mathbf{\Sigma}_{\beta}^{0}\}$.
\item$\mathbf{\Pi}_{\alpha}^{0}=\{X\backslash A:A\in\mathbf{\Sigma}_{\alpha}^{0}\}$.
\item$\mathbf{\Delta}_{\alpha}^{0}=\mathbf{\Sigma}_{\alpha}^{0}\bigcap\mathbf{\Pi}_{\alpha}^{0}$.
\end{enumerate}
for each ordinal $\alpha>0$.

\end{defi}

From this definition, we can derive results analogous to those found in the study of Polish spaces.
Just as $G_\delta$ subspaces of Polish spaces are Polish themselves, $\mathbf{\Pi}_2^0$ subspaces of quasi-Polish spaces are quasi-Polish spaces.
\begin{prop}[{\cite[Theorem 22]{deBr2013}}]\label{1-pi2-quasipolish}

Every $\mathbf{\Pi}_{2}^{0}$ subspace of a quasi-Polish space is quasi-Polish.
\end{prop}

\begin{proof}
For later reference, we repeat the proof below.
Let $Y\subseteq X$ be a $\mathbf{\Pi}^{0}_{2}$ subset represented as $Y=\bigcap_{i\in\omega}(U_{i}\cup A_{i})$, where $U_{i}$ is open, $A_{i}$ is closed,  and $U_{i}\cap A_{i}=\emptyset$ for all $i\in\omega$. Let $d$ be a compatible Smyth-complete quasi-metric on $X$. For each $i\in\omega$, let $F_{i}=X\backslash U_{i}$ and define $d_i:Y\times Y\to[0,\infty)$ as follows:
\[
d_{i}(x,y)=
 \begin{cases}
 \min\{2^{-i},\max\{0,\frac{1}{d(y,F_{i})}-\frac{1}{d(x,F_{i})}\}\}&\text{if }x,y\in U_{i},\\
 2^{-i}&\text{if }x\in U_{i},y\in A_{i},\\
 0&\text{otherwise,}
 \end{cases}
\] where $d(x,F)=\inf_{y\in F}{d(x,y)}$ for every closed set $F$ and point $x\in X\setminus F$.
We shall prove $d^{\prime}=(d+\sum_{i\in\omega}d_{i})$ is a compatible quasi-metric on $Y$.

 \begin{itemize}
  \item \textbf{$d^{\prime}$ satisfies the quasi-metric axioms:} Clearly, $d_{i}(x,y)$ is non-negative for all $i$ and the sum $d^{\prime}(x,y)$ is well-defined as a real number, since it is bounded by $d(x,y)+2$. Since $d$ and each $d_i$ satisfy the triangle inequality, their sum $d^{\prime}$ also satisfies it. Thus, $d^{\prime}$ is a quasi-metric.
  \item \textbf{Compatibility of the topology:} The topology generated by $d^{\prime}$ is equivalent to the relative topology induced by $d$. Since $d^{\prime}(x,y)\geq d(x,y)$ for all $x,y\in Y$, the $d^{\prime}$-topology is at least as fine as the $d$-topology. To show they are equivalent, it suffices to show that for all sufficiently small $\varepsilon>0$ and $x\in Y$, there exists $\delta>0$ such that $\forall y\in Y(d(x,y)<\delta\Rightarrow d^{\prime}(x,y)<\varepsilon)$.
For each $F=F_{i}$ and a sufficiently small $\delta<d(x,F)$, $d(x,y)<\delta$ implies $d(y,F)\geq d(x,F)-\delta$ by the triangle inequality. Then:
\[
\frac{1}{d(y,F)}-\frac{1}{d(x,F)}=\frac{d(x,F)-d(y,F)}{d(x,F)d(y,F)}\leq\frac{\delta}{d(x,F)(d(x,F)-\delta)}
\]
By choosing $\delta$ small enough, this value can be made arbitrarily small. Since $d^{\prime}$ is a sum of $d$ and such terms $d_i$, the claim follows.
  \item \textbf{Completeness of $(Y,d^{\prime})$:} Let $\langle x_k\rangle_{k\in\omega}$ be a left-Cauchy sequence in $(Y,d^{\prime})$. Then $\langle x_k\rangle_k$ is also a left-Cauchy sequence in $(X,d)$. Since $(X,d)$ is Smyth-complete, $\langle x_k\rangle_k$ converges to some point $x\in X$. If $x_k$ is in $A_n$ for infinitely many $k$, then $x\in A_n$ because $A_n$ is closed. Otherwise, $x_k\in U_n$ for all sufficiently large $k$. In this case, the left-Cauchy property of $\langle x_k\rangle_k$ with respect to $d_n$ implies that $\frac{1}{d(x_k,F_n)}$ is bounded, so $d(x,F_{n})>0$, which means $x\in X\backslash F_{n}$. Thus $x\in\bigcap_{n\in\N}(U_n\cup A_n)=Y$.
Finally, we show that $\langle x_k\rangle_k$ converges to $x$ with respect to $d^{\prime}$. For any $n\in\N$, we can choose $K\in\N$ large enough such that for all $k>K$ and $m\leq n$, the terms $d(x_k,x)$ and $d(x,x_k)$ are sufficiently small to satisfy $d^{\prime}(x_k,x)<2^{-n}$ and $d^{\prime}(x,x_k)<2^{-n}$ for large $k$, by the preceding discussion. Thus, $(Y,d^{\prime})$ is Smyth-complete. \qedhere
 \end{itemize}
\end{proof}
The converse is also true.
\begin{prop}[{\cite[Theorem 21]{deBr2013}}]\label{prop:qPisPi20}
Every quasi-Polish subspace of a second-countable quasi-metric space is a $\mathbf{\Pi}_2^0$-subspace. 
\end{prop}

Thus, every quasi-Polish space is homeomorphic to a $\mathbf{\Pi}_2^0$ subspace of $\mc P(\N)$. In addition, second-countability implies separability in the following sense:
\begin{prop}[{\cite[Proposition 13]{deBr2013}}]\label{prop:hat-d-dense}

A quasi-metric space $(X,d)$ is second-countable if and only if the associated metric space $(X,\hat{d})$ is separable.
\end{prop}

\subsection{Countably presented frames}
In this subsection, we examine a specific class of spaces originating from pointless topology, particularly focusing on the property of being ``spatial'' and its relationship with quasi-Polish spaces.

\begin{defi}[frames]
A complete lattice $\tau$ is called a \textit{frame} if it satisfies the following infinite distribution law for every $u \in \tau$ and every family $\langle v_i\rangle_{i \in I} \subseteq \tau$:
\[
u \land \bigvee_{i \in I} v_i = \bigvee_{i \in I} (u \land v_i).
\]
\end{defi}
With every topological space $(X,\tau)$ (where $\tau$ denotes the family of open sets on $X$), we can associate the frame $\tau$. Conversely, a frame $\sigma$ is \textit{spatial} if there exists a topological space $(X,\tau)$ whose frame $\tau$ is isomorphic to $\sigma$ as a frame.

\begin{defi}[countably presented frames]
Let $G$ be a countable set of generators, and let $R \subseteq \langle G \rangle \times \langle G \rangle$ be a countable binary relation. Here, $\langle G \rangle$ denotes the free frame generated by $G$, whose elements are represented in the form $\bigvee_{F \in \mathcal{F}} \bigwedge_{g \in F} g$, where $\mathcal{F}$ is a family of finite subsets of $G$.

A \textit{countably presented frame} $\langle G; R \rangle$ is a frame consisting of a set of generators $G$, a countable binary relation $R \subseteq  \langle G\rangle\times \langle G\rangle$, and the preorder $\preceq$ defined as follows.
\end{defi}

\begin{defi}[congruence preorder]
A binary relation $\preceq_m$ on $\langle G\rangle$ is called a \textit{congruence preorder} containing $R$ if it satisfies the following conditions for all $a, b, c \in \langle G\rangle$:
\begin{enumerate}
\item $a\leq b\Rightarrow a\preceq_m b$, $(a,b)\in R\Rightarrow a\preceq_m b$;
\item $a\preceq_m b$, $b\preceq_m c\Rightarrow a\preceq_m c$.
\end{enumerate} 
In particular, we denote the least congruence preorder by $\preceq$. For two expressions $a, b$ in the countably presented frame $\langle G; R \rangle$, the relation $a \preceq b$ holds if and only if $a \preceq_m b$ holds for every congruence preorder $\preceq_m$.
\end{defi}
For any frame $H$ and any frame homomorphism $\phi: \langle G\rangle \to H$ satisfying $\phi(u) \leq_H \phi(v)$ for all $(u, v) \in R$, the map $\phi$ induces a congruence preorder $\preceq_m$ defined by 
\[
a \preceq_m b :\iff \phi(a) \leq_H \phi(b).
\]
The countably presented frame $\langle G; R \rangle$ can also be defined as the universal object satisfying these requirements.

\[
\xymatrix{
\langle G\rangle \ar[r]^{\iota} \ar[rd]_{\phi} & \langle G;R\rangle \ar@{-->}[d]^{\exists! \tilde{\phi}} \\
& H
}
\]
We can also define $\preceq$ as the minimal transitive closure which contains $R$ and $\leq$. 

Since the set $\{0, 1\}$ with the order $0 < 1$ forms a frame, we can consider a specific class of congruence preorders induced by frame homomorphisms $\xi: \langle G; R \rangle \to \{0, 1\}$. 
Each frame homomorphism corresponds to a point $x$ in $X$ satisfying $x\in g$ iff $\xi(g)=1$ for all $g\in\langle G;R\rangle$, where $X$ is the topological space whose lattice of opensets given by $\langle G;R\rangle$.
For this frame, we define a topological space $X$ where the set of points consists of all frame homomorphisms $\xi: \langle G; R \rangle \to \{0, 1\}$. The topology on $X$ is generated by the subbasis: 
\[
\{ N_g : g \in G \}\cup\{X\} \text{ where }N_g=\{ \xi : \xi(g) = 1 \}.
\] Under this construction, an arbitrary open set $a \in \langle G; R \rangle$, expressed in the form $a = \bigvee_{F \in \mathcal{F}} \bigwedge_{g \in F} g$, corresponds to the open set: 
\[
U_a=\bigcup_{F \in \mathcal{F}} \bigcap_{g \in F} N_g.
\] To verify whether the frame $\langle G; R \rangle$ is spatial, it suffices to check whether the lattice of open sets of this constructed space $X$ is isomorphic to the original frame $\langle G; R \rangle$.

\begin{defi}[Spatial]
A countably presented frame $\langle G; R \rangle$ is said to be spatial if $U_a \subseteq U_b \iff a \preceq b$ holds for all expressions $a$ and $b$. 
\end{defi}

\subsection{Equivalence}
Next, we demonstrate that these different characterizations are equivalent to one another: (1) quasi-Polish spaces, (2) UF spaces, (3) NP spaces, (4) ideal spaces, (5) $\mathbf{\Pi}_2^0$ subspaces of $\mathcal{P}(\mathbb{N})$, and (6) sober spaces associated with countably presented frames. As noted earlier, NP spaces and ideal spaces are equivalent.

\begin{prop}\label{1-UF(P)toP(N)}
\begin{enumerate}
 \item Every countably based UF space is homeomorphic to a $\mathbf{\Pi}^0_2$ subspace of $\mc P(\N)$.
 \item Every $\mathbf{\Pi}^0_2$ subspace of $\mc P(\N)$ is homeomorphic to a countably based UF space.
 \item Every $\mathbf{\Pi}^0_2$-subspace of a countably based UF space is homeomorphic to a countably based UF space.
\end{enumerate}
\end{prop}

\begin{proof}
\begin{enumerate}
\item Let $X$ be a UF space and $P$ be a poset such that $X=\UF(P)$. By Proposition \ref{1-typicalization}, we may assume without loss of generality that $P$ is handy and infinite.
Since $P$ is countably infinite, we enumerate its elements without repetition as $\langle p_i : i \in \mathbb{N} \rangle$. Define a function $g:P\rightarrow \mc P_{\fin}(\N)$ by $g(p_{n})=\{i\leq n:p_{n}\leq p_{i}\}$. Then, for each $F\in \UF(P)$, we define $\hat{g}(F)=\upcl(\{g(p_n):p_n\in F\})$ under the reverse inclusion order on $\mc P_{\fin}(\N)$. This map $\hat{g}:\UF(P)\rightarrow \mc P(\N)$ satisfies the required conditions.

\begin{itemize}
 \item \textbf{$\hat{g}(F)$ is a filter:} 
By definition of $\hat{g}$, $\hat{g}(F)$ is upward closed for every $F\in\UF(P)$.
Let $q,q^{\prime}\in \hat{g}(F)$. There exists $n,n^{\prime}$ such that $p_n,p_{n^{\prime}}\in F$ with $g(p_n)\leq q$ and $g(p_{n^{\prime}})\leq q^{\prime}$. Since $F$ is an unbounded filter and $P$ is handy, there exists $m\in \N$ such that $p_m\in F$, $p_{m}\leq p_{n},p_{n^{\prime}}$ and $m\geq n,n^{\prime}$, which implies $g(p_m)\leq g(p_n),g(p_{n^{\prime}})$. Thus $\hat{g}(F)$ is a filter.   

 \item \textbf{The image of $\hat{g}$ is $\mathbf{\Pi}_2^0$:} Let $\sigma$ be a finite sequence of $P$. We define a property $\phi(n,\sigma)$ which states that the length of $\sigma$ is $n+1$, $\forall i<n(\sigma(i)>\sigma(i+1))$ and $\forall i<n+1(\sigma(n)\ngeq p_{i})$.
 Let $U_{n}$ be the union of  basic open sets $N_{g(\sigma(n))}$ for all $\sigma$ satisfying $\phi(n,\sigma)$. Let $C$ be the $\mathbf{\Pi}^0_2$ subset of $\mc P(\N)$ defined by 
\[
\bigcap_{i,j}(\bigcup_k\{N_k:p_i,p_j\geq p_k\}\cup\mc P(\N)\backslash (N_i\cap N_j))
\] here $N_k=N_{\{k\}}$ and let $B=\bigcap_{n\in\N} U_{n}\cap C$. Since $P$ is handy, it is clear that $\hat{g}(\UF(P))\subseteq B$. Conversely, for any filter $q\in B$, let $F_q=\{p_{i}:i\in\cup q\}$. If $p_{i}\leq p_{j}$ and $i\in\cup q$, there exists $n$ and $\sigma$ such that $q\in N_{g(\sigma(n))}$, $n\geq i,j$ and $\phi(n,\sigma)$ holds. Then $\sigma(n)\leq p_{i}\leq p_{j}$, so $j\in g(\sigma(n))\subseteq\cup q$. If $i,j\in\cup q$, there is $k\in\cup q$ such that $p_i,p_j\geq p_k$ since $q\in C$. Thus $F_q$ is a filter on $P$. 
Furthermore, $F_q$ is non-principal because $\sigma(k)\ngeq p_k$ for all $k$. 
 \item \textbf{Bijection:} For any $F\in \UF(P)$, let $\hat{F}=\{p_{i}:i\in\cup \hat{g}(F)\}$. It follows that $F=\hat{F}$ for all $F$, showing that $\hat{g}$ is an injection onto $B=\bigcap_{n\in\N} U_{n}\cap C$.
 \item \textbf{Continuous and Open mapping:} Since $\hat{g}$ is induced by $g$, $\hat{g}$ is an open mapping. Furthermore, for $q\in\mc P_{\fin}(\N)$, $\hat{g}^{-1}(N_{q})=\{F\in \UF(P):\forall i\in q(p_{i}\in F)\}$ which is open in $\UF(P)$. This confirms that $\hat{g}$ is continuous.
\end{itemize}

\item  Let $X\subseteq\mc P(\N)$ be a nonempty $\mathbf{\Pi}_2^0$-subspace defined by $X=\bigcap_{n\in\N}(B_n\cup A_n)$, where each $A_n$ is a closed subset and each $B_n$ is an open subset. We define a poset $P^{\prime}:=\{p_{n,q}:n\in\N,q\in\mc P(\N)\text{ is finite, }q\in \bigcap_{i=0}^{n}(B_{i}\cup A_i)\}$. The order on $P^{\prime}$ is defined by $p_{n,q}\leq p_{m,r}\Leftrightarrow (n>m\land q\supseteq r)\lor(n=m\land q=r)$. Let  $P$ be the union of subposets of $P^{\prime}$ in which every element has a strictly smaller element. Then $P$ is handy. Indeed, every unbounded filter is non-principal since $P$ has no minimal element, and every non-principal filter is unbounded since $p_{n,q}<p_{m,r}$ implies $n>m$. 
We define the following maps, which we will show are continuous and mutually inverse.
\[
f:\UF(P)\rightarrow X, f(F)=\{i\in\N:\exists p_{n,q}\in F (i\in q)\}
\]
\[
g:X\rightarrow \UF(P), g(G)=\upcl\{p_{n,q}\in P:q\subseteq G\text{ and $q$ is finite}\}
\]
 \begin{itemize}
  \item \textbf{Well-definedness:}
Suppose for contradiction that $f(F)\notin X$. Then there exists $n\in\N$ such that $f(F)\notin A_n\cup B_n$. Then every $p_{m,q}\in F$ satisfies $q\notin B_n$.
Since $A_n$ is closed, $f(F)\notin A_n$ implies there is a finite subset $q\subseteq f(F)$ such that $q\notin A_n$. Since $F$ is an unbounded filter, there exists $p_{m,q^{\prime}}\in F$ such that $q\subseteq q^{\prime}$ and $m\geq n$. Since $f(F)\notin B_n$ and $q^{\prime}\subseteq f(F)$, it follows that $q^{\prime}\notin B_n$. Thus $q^{\prime}\notin A_n\cup B_n$ which contradicts the fact that $p_{m,q^{\prime}}\in P$. Therefore, $f(F) \in X$.
Conversely, let $G\in X$. For every $n$, since $G\in \bigcap_{n\in\N}(B_n\cup A_n)$, there exists a finite $q\subseteq G$ such that $p_{n,q}\in P^{\prime}$. Thus $g(G)$ is unbounded and $p_{n,q}\in P$. To show $g(G)$ is a filter, let $p_{m,q},p_{n,q'}\in g(G)$. Since $G\in X$, for some $N>n,m$, there exists a finite $r\subseteq G$ such that $p_{N,r}\in P$ and $G\in B_{i}\to r\in B_{i}$ for each $i\leq N$ . Since $q,q',r\subseteq G\in X$, we have $q\cup q'\cup r\in\bigcap^{N}_{i=0}(A_i\cup B_i)$, $p_{N, q\cup q'\cup r}\in g(G)$ and $p_{N, q\cup q'\cup r}\leq p_{m,q},p_{n,q'}$. Thus $g(G)$ is a filter.   
  \item \textbf{Inverse maps:} Given $G\in X$, it is clear that $f(g(G))\subseteq G$. For the reverse inclusion, let $i\in G$. Since $G\in A_0\cup B_0$, there exists a finite $q\subseteq G$ containing $i$ such that $p_{0,q}\in P$. Specifically, if $G\in A_0$, then $q=\{i\}\in A_0$; if $G\in B_0$, there is a finite $q^{\prime}\subseteq G$ in $B_0$, so $q^{\prime}\cup\{i\}\in B_0$. Thus $p_{0,q}\in g(G)$, so $i\in f(g(G))$.
Now let $F\in \UF(P)$. Clearly $F\subseteq g(f(F))$. Conversely, if $p_{n,q}\in g(f(F))$, then $q\subseteq f(F)$. For each $i\in q$, choose $p_{m_i,q_i}\in F$ such that $i\in q_i$. Since $q$ is finite and $F$ is a filter, there exists $p_{m,q^{\prime}}\in F$ such that $\forall i(p_{m_i,q_i}\geq p_{m,q^{\prime}})$. Then $q\subseteq q^{\prime}$. Since $F$ is unbounded, we can find $p_{N,q^{\prime\prime}}\in F$ such that $N> n$ and $q^{\prime}\subseteq q^{\prime\prime}$. Then $p_{n,q}> p_{N,q^{\prime\prime}}$, and hence $p_{n,q}\in F$ since $F$ is upward closed.
  \item \textbf{Continuity:} The map $f$ is continuous since $f^{-1}(N_i)=\bigcup \{ N_{p_n,q} : q\in \mc P_{\fin}(\N) \land i \in q\}$. The map $g$ is continuous since $g^{-1}(N_{p_{n,q}})=\{G\in X:q\subseteq G\}$, which is equal to $\bigcap_{i\in q}N_i\cap X$.
 \end{itemize}
The statement is proved.

\item Straightforward consequence of 1 and 2. \qedhere
\end{enumerate}
\end{proof}

The class of countably based UF spaces coincides exactly with the class of $\mathbf{\Pi}_2^0$ subspaces of $\mc P(\N)$. Hence, we have the following equivalence, which was first pointed out by Tadayuki Honda \cite{honda-unpublished}.

\begin{thm}[Essentially, Honda \cite{honda-unpublished}]
Every countably based UF space is quasi-Polish.
\end{thm}
\begin{proof}
It follows by Proposition \ref{1-UF(P)toP(N)} and Proposition \ref{1-pi2-quasipolish}.
\end{proof}

\begin{thm}[Essentially, Honda \cite{honda-unpublished}]
Every quasi-Polish space is a countably based UF space. 
\end{thm}
\begin{proof}
It follows by Propositions \ref{1-UF(P)toP(N)}, \ref{prop:qPisPi20} and the properties of $\mc P(\N)$.
\end{proof}
This theorem can be seen as a generalization of the result by Mummert/Stephan~\cite{mummert2010topological} from Polish spaces to quasi-Polish spaces. Its second-order arithmetic counterpart, Theorem \ref{qPtoUF}, is obtained by following the approach of their Theorem 2.3.

So every countably based UF space is separable with the associated metric $\hat{d}$. On the contrary, every UF space whose associated metric space $(X,\hat{d})$ is separable is countably based, given that $(X,d)$ is second-countable by Proposition \ref{prop:hat-d-dense}.

\begin{cor}
Every subspace of a UF space which is homeomorphic to a UF space is a $\mathbf{\Pi}_2^0$-subspace.
\end{cor}

Although the following corollary is an immediate consequence of Propositions~\ref{1-typicalization}.3 and \ref{prop:1-NPtoNPUF}.1, we also provide an alternative proof below.

\begin{cor}\label{1-NP-typicalization}
Let $P$ be a countable poset.
There exists a handy poset $P^{\prime}$ such that $\UF(P^{\prime})$ is homeomorphic to $\NP(P)$.
\end{cor}
\begin{proof}
By applying the construction $R \mapsto \hat{R}$ from Proposition~\ref{1-typicalization}.3 starting with $P_0 = P$, we obtain $P_{\alpha}$ such that every unbounded filter in $P_{\alpha}$ is non-principal.
We define a map $\hat{g}$ from $\NP(P_{\alpha})$ to a subspace of $\mc P(\N)$ induced by the function $g:P\rightarrow\mc P_{\fin}(\N)$ defined as:
\[
g(p_n) :=\{i\leq n:p_n\leq p_i\}
\]
where $\langle p_n:n\in\N\rangle$ is an enumeration of $P_{\alpha}$. As in the proof of Proposition \ref{1-UF(P)toP(N)}.1, one can verify that $\hat{g}$ is an injective, continuous, and open mapping.
Next, we define a property $\psi(n,\sigma,r,p)$ indicating that: the length of $\sigma$ is $n+1$; $\forall i<n(\sigma(i)>\sigma(i+1))$; $\sigma(n)=p$; and $(\forall i\in\{0,...,n\}\backslash r)\,(p_i\not\leq p)$ where $n\in\N$, $\sigma$ is a finite sequence in $P_{\alpha}$, $r\subseteq\{0,...,n\}$, $p\in P_{\alpha}$. For each $n$ and $r$, there exists a $\mathbf{\Delta}_2^0$-subspace $D_{n,r}$ of $\mc P(\N)$ defined by:
\[
D_{n,r}=(\bigcup_{\exists\sigma\psi(n,\sigma,r,p)}N_{g(p)})\cap\bigcap_{i\in r}M_{g(p_i)}
\] where $M_p$ is a complement of $N_p$.
Let $U_n^{\prime}=\bigcup_{r\subseteq\{0,...,n\}}D_{n,r}$. Each $U_n^{\prime}$ is also a $\mathbf{\Delta}_2^0$ subspace. 
Let $C$ be the $\mathbf{\Pi}^0_2$ subset of $\mc P(\N)$ defined by:
\[
\bigcap_{i,j}(\bigcup_k\{N_k:p_i,p_j\geq p_k\}\cup\mc P(\N)\backslash (N_i\cap N_j)).
\]
We then set $B^{\prime}=\bigcap_{n\in\N}U^{\prime}_n\cap C$. It can be shown that $\hat{g}(\NP(P_{\alpha}))=B^{\prime}$. Indeed, $\hat{g}(F)\in U_n^{\prime}$ for every $n\in\N$ and $F\in \NP(P)$. Conversely, if $q\in B^{\prime}$, then $F=\{p_i:i\in\cup q\}$ is a filter due to the definition of $g$ and the fact that $q\in C$. $F$ is non-principal because if there were some $k\in\N$ such that $p_k\in F$ and $\forall p\in F(p_k\leq p)$, this would contradict the definition of $D_{k,r}$ (specifically, the condition $(\forall i\in\{0,...,n\}\backslash r)\,(p_i\not\leq p)$).
\end{proof}

Finally, we will see that countably presented frames form families of open sets of quasi-Polish spaces.
\begin{thm}[Heckmann \cite{heckmann2015spatiality}]
Every countably presented frame $\langle G; R \rangle$ is spatial. That is, there exists a topological space whose frame of open sets is isomorphic to $\langle G; R \rangle$.
\end{thm}

The point space can be described as a subset of $\mathcal{P}(G)$ in the following form:
\[
\bigcap_{(u,v) \in R} ((\mathcal{P}(G) \setminus U) \cup V),
\]
where $U$ and $V$ are the open sets corresponding to the relations $(u,v)\in R$. Since this set is a countable intersection of open sets (as $R$ is countable), the point space is homeomorphic to a $\mathbf{\Pi}_2^0$ subspace of $\mathcal{P}(\mathbb{N})$.

\subsection{Some more properties of quasi-Polish spaces}
Here we see properties of the class of quasi-Polish spaces. Some of these properties are direct consequences of the equivalence of representations. 

\begin{prop}[{\cite[Theorem 41]{deBr2013}}]
A $T_0$ second-countable space is quasi-Polish if and only if it has an open continuous surjection from $\N^\N$.
\end{prop}

\begin{prop}[{\cite[Theorem 5.2]{mummert2010topological}}]
Every poset space has the strong Choquet property.
\end{prop}

\begin{prop}
Let $X$ be a UF space. For every point $x \in X$, there exists a point $y \in X$ such that $x$ is in the closure of $\{y\}$ and $\{y\}$ is a $G_\delta$ subset.
\end{prop}
\begin{proof}
For every filter $F$, there is a maximal filter $G$ such that $F\subseteq G$ by Zorn's lemma.
\end{proof}

The next proposition answers Question 4.5 of \cite{mummert2010topological}.
Indeed, it is a direct consequence of the fact that the product of two quasi-Polish spaces is quasi-Polish and Proposition~\ref{1-UF(P)toP(N)}, but here we give a direct proof.
\begin{prop}\label{1-product-UF}
For every pair of countable posets $P,Q$, there exists a countable poset $R$ such that $\UF(P)\times \UF(Q)$ is homeomorphic to $\UF(R)$. Similarly, for every countable sequence of countable posets $\langle P_n:n\in\N\rangle$, there exists a countable poset $R$ such that $\prod_{n\in\N}\UF(P_n)$ is homeomorphic to $\UF(R)$. 
\end{prop}
\begin{proof}
We may assume without loss of generality that $P,Q$ are handy posets. Let $P=\langle p_{i}:i\in\N \rangle$ and $Q=\langle q_{j}:j\in\N \rangle$. We define $R=P\times Q$ with the order: 
\[
(p,q)\leq(p^{\prime},q^{\prime})\Leftrightarrow (p<p^{\prime}\land q<q^{\prime})\lor(p=p^{\prime}\land q=q^{\prime}).
\] This $R$ satisfies the requirements. Indeed, if $F$ and $G$ are unbounded filters of $P$ and $Q$, respectively, then $f(F,G):=\{(p,q)\in R:p\in F\land q\in G\}$ is an unbounded filter of $R$. It is straightforward to verify that $f$ is an embedding (a homeomorphism onto its image). Conversely, for any unbounded filter $H$ of $R$, let $g_{1}(H):=\{p\in P:\exists q\in Q\,(p,q)\in H\}$ and $g_{2}(H):=\{q\in Q:\exists p\in P\,(p,q)\in H\}$. These sets $g_{1}(H),g_{2}(H)$ are unbounded filters of $P$ and $Q$ respectively, and we have $f(g_{1}(H),g_{2}(H))=H$, which shows $f$ is surjective. The case for the countable product $\prod_{i\in\N}\UF(P_i)$ follows a similar argument. 
We can take the poset $R$ as follows:
\[
R:=\{\langle p_i:i<k\rangle:k\in\N, \forall i<k(p_i\in P_i)\}
\]
 with the ordering $\leq$ such that
\begin{align*}
\langle p_i:i<k\rangle\leq\langle q_i:i<h\rangle\iff& (k>h\land \forall i<k\,(p_i<q_i))\\
&\lor(k=h\land \forall i<k(p_i=q_i)).
\end{align*}
\end{proof}

\begin{cor}
There exists a UF space that is Hausdorff but not regular.
\end{cor}
\begin{proof}
Let $\R$ be the set of real numbers equipped with the usual topology. We define $K=\{\frac{1}{n}:n\in\N\}$ and let $B=\R\backslash K$. Note that $B$ is not closed in the usual topology but is a $\mathbf{\Delta}_{2}^{0}$ set.
Since $\R$ is a quasi-Polish space, the topology generated by the usual open sets together with the set $B$ also yields a quasi-Polish space \cite[lem71]{deBr2013}. This resulting topology is known as the K-topology.
In the K-topology, the space remains Hausdorff because it is finer than the usual Euclidean topology. However, it is not regular: the set $K$ is closed in this topology, but the point $0$ and the closed set $K$ cannot be separated by disjoint open sets. Since every second countable quasi-Polish space is homeomorphic to a UF space by Proposition \ref{1-UF(P)toP(N)}, there exists a UF space with these properties.
\end{proof}
This corollary resolves Open Problem 2.3.8 of \cite{mummert2005reverse}.

\section{Formalization in second-order arithmetic}
In this section, we formalize the results of the previous section within the framework of second-order arithmetic. For details on the subsystems of second-order arithmetic, we refer the reader to the standard textbook by Simpson \cite{SOSOA}.
We use the following subsystems of second-order arithmetic in this paper.
$\RCAo$ is the base theory, which consists of the basic axioms of first-order arithmetic, $\Delta_1^0$-comprehension, and $\Sigma_1^0$-induction. $\ACAo$ extends $\RCAo$ by adding arithmetical comprehension. $\PCAo$ consists of $\RCAo$ plus $\Pi_1^1$-comprehension.

Our analysis primarily concerns second-countable $T_0$ spaces, for which we develop codes for Borel sets and continuous maps. We will use codes that generalize the codes of Polish spaces introduced by Simpson \cite{SOSOA}. 

Within second-order arithmetic, we view a second-countable $T_0$ space $X$ as a structure consisting of:
\begin{itemize}
  \item a class of points $X$, where each point $x$ is typically represented as a second-order object.
Then, a space $X$ is represented by a predicate $P_X(x)$ (an $\mc L_2$-formula) and we denote $x\in X$ when $P_X(x)$ holds;
  \item an equivalence relation $=_X$ on $X$ representing identity of points (if needed);
  \item a countable basis coded by $W \subseteq \mathbb{N}$; and
  \item a membership relation $x \in N_w$ for $x \in X$ and $w \in W$, indicating that $x$ belongs to the basic open set $N_w$ coded by $w$.
\end{itemize}
These components are required to satisfy the $T_0$ separation axiom:
\[
\forall x, y \in X \left( x \neq_X y \implies \exists w \in W \, ((x \in N_w \land y \notin N_w) \lor (x \notin N_w \land y \in N_w)) \right).
\]
 For a second-countable $T_0$ space $X$ with a fixed basis code $W\subseteq\N$, a code for an open set is defined as a subset $U \subseteq \mathbb{N} \times W$.
Given the membership relation $x\in N_w$ for each $x\in X$ and $w\in W$, the code $U$ represents the open set $O_U:=\bigcup_{(n,w)\in U}N_w=\{ x \in X : \exists (n, w) \in U \,(x \in N_w) \}$.
Note that the above definition of open sets is justified only when the coded basis satisfies the condition for open bases (closed under finite intersections), i.e., 
\begin{align*}
\forall w,w'\in W\,\forall x\in X(x\in N_w\land x\in N_{w'}\rightarrow &\exists v\in W(x\in N_v\\
&\land \forall y\in X(y\in N_v\to y\in N_w\land y\in N_{w'}))).
\end{align*}
The spaces used in this section naturally satisfy this condition. 

As in the previous section, a subspace represented by an open set code (or the empty set) is called a \textit{$\mathbf{\Sigma}_1^0$ subspace}. From this, Borel sets are naturally defined.

\begin{definition}[Borel sets of finite ranks]\label{defi:Borel-sets}
Let $X$ be a second-countable $T_0$ space with a countable basis code $W\subseteq\N$. 
\begin{enumerate}
    \item The Borel code for the empty set is defined as $\{(1, 1, \langle\rangle)\}$.
    \item For an open set $U \subseteq \mathbb{N} \times W$, its Borel code is $(1, 1, \langle w_i : i \in \mathbb{N} \rangle)$, where $\langle w_i : i \in \mathbb{N} \rangle$ is an enumeration of $\{w\in W:\exists n((n,w)\in U)\}$.
    \item For $1 < k < \omega$, a \textit{$\mathbf{\Sigma}_k^0$ subspace} $Y = \bigcup_{i \in \mathbb{N}} (B_i \setminus B^{\prime}_i)$ (where $B_i, B^{\prime}_i$ are $\mathbf{\Sigma}_{k-1}^0$ subspaces coded by $(1, k-1, \langle v_{i,m}:m\in\N \rangle)$ and $(1, k-1, \langle w_{i,m}:m\in\N \rangle)$ respectively) is coded as:
    \[
    (1, k, \langle (1, v_{i,m}), (0, w_{i,m}) : i, m \in \mathbb{N} \rangle).
    \]
    \item The \textit{$\mathbf{\Pi}_k^0$ subspace} $X \setminus Y$, the complement of $Y$, is coded as:
    \[
    (0, k, \langle (0, v_{i,m}), (1, w_{i,m}) : i, m \in \mathbb{N} \rangle).
    \]
\end{enumerate}
For example, a $\mathbf{\Pi}_2^0$ subspace $\bigcap_{i \in \mathbb{N}} (A_i \cup B_i)$ (where $B_i = \bigcup_{n \in \mathbb{N}} N_{v_{i,n}}$ is an open set coded by $(1,1,\langle v_{i,n}:n\in\N\rangle)$ and $A_i = X \setminus \bigcup_{n \in \mathbb{N}} N_{w_{i,n}}$ is a closed set coded by $(0,1,\langle w_{i,n}:n\in\N\rangle)$) is coded by the sequence:
\[
(0, 2,\langle (0, w_{i,n}), (1, v_{i,n}) : i, n \in \mathbb{N} \rangle) \subseteq \{(0,2)\} \times (\{0\}\times W \times \{1\} \times W) \times \mathbb{N}.
\]
\end{definition}

\begin{defi}[Continuous map]
Let $X$ and $Y$ be countably based spaces, and let $Z$ and $W$ denote the sets of codes for their respective basic open sets. A code for a continuous map $\phi$ from $X$ to $Y$ is a set $\Phi\subseteq\N\times Z\times W$. For $w\in W$, we define the code of the inverse image $\phi^{-1}(N_w)$ to be the set  $\{(n,z)\in \N\times Z:\exists m<n\, (m,w,z)\in\Phi\}$. Consequently, for a code $U\subseteq\N\times W$ of an open set in $Y$, the inverse image $\phi^{-1}(O_U)$ is given by:
\begin{align*}
&\bigcup\{\phi^{-1}(N_w):(n,z,w)\in\Phi,(m,w)\in U\}\\
=&\{x\in X:\exists n,m\,\exists z\in Z\,\exists w\in W(x\in N_z\land(n,z,w)\in\Phi\land (m,w)\in U)\}.
\end{align*} 
For $x\in X$ and $y\in Y$, $\phi(x)=y$ if: 
\[
\forall w\in W(\exists n\exists z\in Z(x\in N_z\land (n,z,w)\in\Phi)\leftrightarrow (y\in N_w)).
\]
Since $Y$ is a $T_0$ space, such a $y$ is unique (if it exists).
 A point $x\in X$ is in the domain of $\Phi$ (denoted $x\in \mathrm{dom}(\Phi)$) if $\exists y\in Y\,\forall w\in W(y\in N_w\leftrightarrow x\in \phi^{-1}(N_w))$. In the case of $T_0$ spaces, the condition $(y \not\in N_w) \rightarrow \neg \exists(n,z,w)\in\Phi (x\in N_z)$ cannot be omitted. 

A continuous map $\phi:X\rightarrow Y$ is said to be total if its code $\Phi$ satisfies $\forall x(x\in X\rightarrow x\in \text{dom}(\Phi))$.
\end{defi}

\begin{prop}[Codes for continuous maps]
In the following sense, the definition above of codes $\Phi$ for continuous maps between second-countable $T_0$ spaces generalizes the notion of codes $\Psi$ for continuous maps between Polish spaces defined by Simpson \cite{SOSOA} within $\mathsf{RCA}_0$:
\begin{enumerate}
  \item Every code $\Psi$ of a map $\phi$ can be naturally viewed as a code $\Phi$ such that $\mathrm{dom}(\Phi) = \mathrm{dom}(\Psi)$ and the relation $\phi(x) = y$ defined by $\Phi$ coincides with that defined by $\Psi$.
  \item Conversely, from any code $\Phi$, one can construct a code $\Psi$ that satisfies the required conditions for codes at least on $\mathrm{dom}(\Phi)$ such that $\mathrm{dom}(\Phi) \subseteq \mathrm{dom}(\Psi)$, where the relation $\phi(x) = y$ defined by $\Psi$ coincides with that defined by $\Phi$ for all $x \in \mathrm{dom}(\Phi)$.
\end{enumerate}
\end{prop}

\begin{proof}
When $X$ and $Y$ are Polish spaces with codes $(A,d)$ and $(B,d')$, respectively, codes for continuous maps were introduced by Simpson \cite[Definition~II.6.1]{SOSOA} as follows. For a continuous map $\phi: X \to Y$, a code $\Psi \subseteq \mathbb{N} \times A \times \mathbb{Q}_+ \times B \times \mathbb{Q}_+$ for $\phi$ satisfies the following conditions:
\begin{enumerate}
  \item If $(n,a,r,b,s) \in \Psi$ and $(m,a,r,b',s') \in \Psi$, then $d'(b,b') \leq s+s'$;
  \item If $(n,a,r,b,s) \in \Psi$ and $d(a,a')+r' < r$, then $(k,a',r',b,s) \in \Psi$ for some $k \in \mathbb{N}$;
  \item If $(n,a,r,b,s) \in \Psi$ and $d'(b,b')+s < s'$, then $(k,a,r,b',s') \in \Psi$ for some $k \in \mathbb{N}$.
\end{enumerate}
Under this definition, we write $x \in \mathrm{dom}(\Psi)$ if $\forall h \in \mathbb{Q}_+ \, \exists (n,a,r,b,s) \in \Psi \, (d(x,a) < r \land s < h)$. Furthermore, we define $y = \phi(x)$ with respect to $\Psi$ if $\forall (n,a,r,b,s) \in \Psi \, (d(x,a) < r \to d'(y,b) \leq s)$.

We can define a code for a continuous map $\Phi=\Psi \subseteq \mathbb{N} \times A \times \mathbb{Q}_+ \times B \times \mathbb{Q}_+$. Then for all $x \in X$, $x \in \mathrm{dom}(\Phi)$ if and only if $x \in \mathrm{dom}(\Psi)$, and the values of $\phi(x)$ correspond between $\Phi$ and $\Psi$ within $\mathsf{RCA}_0$.

Conversely, if we have a code for a continuous map $\Phi$, we can define $\Psi \subseteq \mathbb{N} \times A \times \mathbb{Q}_+ \times B \times \mathbb{Q}_+$ within $\mathsf{RCA}_0$ by
\[
\Psi := \left\{ (n,a,r,b,s) \in \mathbb{N} \times A \times \mathbb{Q}_+ \times B \times \mathbb{Q}_+ : 
\begin{array}{l}
\exists (m,c,h,e,k), (m',c',h',e',k') < n \\
((m,c,h,e,k), (m',c',h',e',k') \in \Phi \\
\land d(c,a)+r < h \land d'(e',b)+k' < s)
\end{array}
\right\}.
\]
Then $\Psi$ satisfies Conditions 2 and 3. We can prove that $x \in \mathrm{dom}(\Phi) \to x \in \mathrm{dom}(\Psi)$ for any $x \in X$. Moreover, if there exist $(n,a,r,b,s) \in \Psi$ and $(m,a,r,b',s') \in \Psi$ such that $d'(b,b') > s+s'$ (i.e., Condition 1 is violated), then $x \notin \mathrm{dom}(\Phi)$ for all $x \in X$ satisfying $d(x,a) < r$. This operation is available within $\mathsf{RCA}_0$.
\end{proof}

In Simpson's definition \cite{SOSOA}, the domain $\mathrm{dom}(\Psi)$ of a code $\Psi$ for a continuous map between Polish spaces is a $G_\delta$ subset given by
\[
\mathrm{dom}(\Psi) = \bigcap_{n \in \mathbb{N}} \bigcup \{ N_{(a,r)} : \exists (n,a,r,b,s) \in \Psi \land s < 2^{-n} \}.
\]
For general second-countable $T_0$ spaces, $\mathrm{dom}(\Phi)$ is not necessarily a countable intersection of open sets. However, if both $X$ and $Y$ are quasi-Polish spaces or UF spaces, $\mathrm{dom}(\Phi)$ forms a $\mathbf{\Pi}_2^0$ subspace within a sufficiently strong subsystem of second-order arithmetic.

\begin{rem}
We deal with second-countable $T_0$ spaces such as poset spaces and quasi-Polish spaces. Although their concrete representations differ from one another, they can all be formalized as second-countable $T_0$ spaces within second-order arithmetic. Consequently, we can treat codes for their continuous maps in a uniform manner.
\end{rem}

\subsection{Poset spaces and ideal spaces}
In what follows, we will work within $\RCAo$. Hence we will mainly deal with countable objects. 
Within $\RCAo$, a (countable) poset is defined as a pair $(P,\leq_P)$ where $P\subseteq\N$ and $\leq_P$ is a preorder on $P$.
%
\begin{defi}[filter as a set] 
\begin{itemize}
\item A subset $F$ of $P$ is called a filter if it satisfies the following conditions:
\begin{enumerate}
\renewcommand{\labelenumi}{(\roman{enumi})}
 \item$\forall p,p^{\prime}\in P((p\in F\land p\leq_P p^{\prime})\rightarrow p^{\prime}\in F)$;
 \item
$\forall p,p^{\prime}\in P((p,p^{\prime}\in F)\rightarrow\exists r\in F(p,p^{\prime}\geq_P r))$;
 \item
$F\neq\emptyset$.
\end{enumerate}

\item A filter $F$ is unbounded if $\lnot\exists r\in P(\forall p\in F(r<_Pp))$.

\item A filter $F$ is maximal if it is not a proper subset of any other filter. For instance, if a filter $F$ satisfies $\forall r\in P\,((r\notin F)\rightarrow\exists p\in F\,(r\bot p))$, then $F$ is clearly a maximal filter. Here $r\bot p$ means that $r$ and $p$ are incompatible (i.e., they have no common lower bound in $P$).

\item A filter $F$ is non-principal if $\forall p\in F\,\exists p^{\prime}\in F(p>_Pp^{\prime}\land p\nless_P p^{\prime})$.

\item We abbreviate the statement that a set $x$ is a filter on $P$ as $x\in\mathbb{F}(P_{\mr set})$. Similarly, we denote unbounded, maximal, and non-principal filters as $x\in\UF(P_{\mr set}),x\in\MF(P_{\mr set})$, and $x\in\NP(P_{\mr set})$ respectively.
\item A \textit{code of an open set} $U$ in $\mathbb{F}(P_{\mr set})$ is defined as a set $U\subseteq\N\times P$. A point (i.e., a filter) $x\in\mathbb{F}(P_{\mr set})$ belongs to $U$, denoted by $x\in U$, if $\exists n\exists p((n,p)\in U\land p\in x)$. 
Then we define Borel sets of finite ranks as in Definition~\ref{defi:Borel-sets}.
\item For a subset $K$ of a poset $P$, the \textit{upward closure} of $K$, denoted $\upcl(K)$, is the set $\{p\in P:\exists q\in K(q\leq_P p)\}$. For a sequence $\langle a_i\rangle_{i\in\N}\subseteq P$, we also write $\upcl(\langle a_i\rangle_{i\in\N})=\{p\in P:\exists i(a_i\leq_P p)\}$. (Note that $\upcl(K)$ may not exist as a set within $\RCAo$.) 
\end{itemize}
\end{defi}

While the previous definition treats a filter as a set, the following proposition demonstrates that decreasing sequences are often preferable to filters within the framework of $\RCAo$.

\begin{prop}\label{existence-upcl}
The following statement is equivalent to $\ACAo$ over $\RCAo$: ``For every subset $Q$ of a poset $P$, the upward closure of $Q$ exists''.
\end{prop}
\begin{proof}
(cf.\ \cite[Theorem 4.1.2]{mummert2005reverse} for an alternative proof)
We first observe that $\ACAo$ implies the statement. Indeed, if $P$ is a poset and $Q\subseteq P$ then the upward closure of $Q$ can be formed using arithmetical comprehension.

Conversely, we show that the statement implies $\ACAo$.
It is well-known that $\ACAo$ is equivalent to the assertion that the range of every injective function $f:\N\to\N$ exists as a set. Let $f:\N\to\N$ be an injective function. We define a poset $P$ on the set ($\{0,1\}\times\N$) with the following partial order:
\[
\begin{cases}
(0,m)\geq(0,n)&\text{ if }m\leq n\\
(1,n)\geq (0,m)&\text{ if }\exists i\leq m(f(i)=n).
\end{cases}
\]
Now, let $Q=\{(0,m):m\in\N\}$. By our assumption, the upward closure $F=\upcl(Q)$ exists as a set. Then the set $F^{\prime}=\{n:(1,n)\in F\}$ is exactly the range of $f$. Indeed, $(1,n)\in F$ if and only if there exists some $(0,m)\in Q$ such that $(1,n)>(0,m)$, which by definition occurs if and only if $n$ is in the range of $f$. Thus, the range of $f$ exists, which implies $\ACAo$.
\end{proof}

Conversely, from a filter as a set, it is easier to obtain one represented a decreasing sequence.  
Indeed, every $\Sigma^{0}_{1}$-definable filter can be described by a decreasing sequence in the following sense.
\begin{prop}
The following is provable in $\RCAo$. Let $P$ be a countable poset and $f:\N\to P$ be a function such that $F=f(\N)$ forms a filter on $P$. Then there exists a decreasing sequence $\langle a_n\rangle_{n\in\N}$ such that $F=\upcl(\langle a_n\rangle_n)$.
\end{prop} 
\begin{proof}
Let $P\subseteq\N$ be a countable poset and let $f:\N\to P$ be such that $F=f(\N)$ forms a filter on $P$. We construct the sequence $\langle a_n\rangle_{n\in\N}$ by primitive recursion. Let $a_{0}=f(0)$. For the inductive step, we define $a_{n+1}=f(j)$ where $j\in\N$ is the least number such that  $(a_n\geq _P f(j))\land \A i\le n(f(j)\leq_{P} f(i))$. By this construction, the inclusion $\upcl(\langle a_n\rangle_{n})\subseteq F$ is immediate. Conversely, since $F$ is a filter, for any $p=f(i)\in F$, we have $a_{i+1}\leq_P p$, which implies $F\subseteq\upcl(\langle a_n\rangle_{n})$.
\end{proof}

Hence, it may be more natural to consider poset spaces consist of filters given by decreasing sequences. When defining a poset space with a preorder $(P, \leq_P)$, we require that each decreasing sequence $\langle a_n\rangle_{n\in\N}$ satisfies $\forall p\in P((p<_P a_n\land a_n<_Pp)\rightarrow a_n<_\N p)$ for each $n$. This modification does not affect the resulting filters and can be formalised within $\RCAo$. 
\begin{defi}[filter as a decreasing sequence] 
\begin{itemize}
\item A decreasing sequence $F=\langle p_{i}: i\in\N\rangle$ in $P$ is called a filter. For every decreasing sequence $F=\langle p_{i}: i\in\N\rangle$, its upward closure $\upcl(F)=\upcl(\langle p_{i}: i\in\N\rangle)$ forms a filter as a set. Two filters $F=\langle p_i:i\in\N\rangle$ and $F^{\prime}=\langle q_j:j\in\N\rangle$ are equivalent as filters, denoted by $F=_{\mathbb{F}(P)} F'$, if and only if $\forall i\exists j (p_i\geq_P q_j)$ and $\forall j\exists i (q_j\geq_P p_i)$. 
\item A filter $F$ is unbounded (resp.~maximal, non-principle) if $\upcl(F)$ is unbounded (resp.~maximal, non-principle).
\item We abbreviate the statement that $F$ is a filter on $P$ as $F\in\mathbb{F}(P_{\mr seq})$. Similarly, we denote unbounded, maximal, and non-principal filters as $F\in\UF(P_{\mr seq}),F\in\MF(P_{\mr seq})$, and $F\in\NP(P_{\mr seq})$ respectively.
\item A \textit{code of an open set} $U$ in $\mathbb{F}(P_{\mr seq})$ is defined as a set $U\subseteq\N\times P$. A point $x=\langle p_{i}: i\in\N\rangle\in\mathbb{F}(P_{\mr seq})$ belongs to $U$, denoted by $x\in U$, if $\exists (n,q)\in U\, \exists i \in \N (p_i \leq_P q)$. 
Then we define Borel sets of finite ranks as in Definition~\ref{defi:Borel-sets}.
\item A poset $P$ is said to be \textit{handy} if $\UF(P_{\mr seq})=\NP(P_{\mr seq})$ and $\UF(P_{\mr seq})$ is a $G_{\delta}$ subset of $\mathbb{F}(P_{\mr seq})$ (there is a sequence $(U_{n})_{n}$ of open sets of $\mathbb{F}(P_{\mr seq})$ such that $\bigcap_{n\in\mathbb{N}}U_{n}=\UF(P_{\mr seq})$). The equality $\UF(P_{\mr seq})=\NP(P_{\mr seq})$ means that a filter in $P$ is unbounded if and only if it is non-principal.
\item A poset $P$ is said to be \textit{pruned} if every $p\in P$ has an element $q\in P$ such that $p>_Pq \land  p\nless q$. This is clearly equivalent to the condition $\UF(P_{\mr seq})\subseteq\NP(P_{\mr seq})$ within $\RCAo$.
\end{itemize}
Note that while every handy poset is pruned, the converse does not necessarily hold; that is, not every pruned poset is handy.
\end{defi}

The choice of representation for filters is intrinsically linked to their complexity.

\begin{prop}\label{UFfml}
Let $P$ be a countable poset
and let $\phi(x)$ be a $\mathbf{\Sigma}_n^0$-formula for $n\in\omega$ such that $(x=y\land\phi(x))\rightarrow\phi(y)$ for all $x,y\in\UF(P_{\mr seq})$ (resp. $\UF(P_{\mr set})$.
The following complexity bounds hold within $\RCAo$, depending on the representation of $x$:
\begin{enumerate}
\item If $x \in \UF(P_{\mr set})$ is coded as a set, then there exists a $\mathbf{\Sigma}_{n+1}^0$-subset $U \subseteq \UF(P_{\mr set})$ such that $x \in U \leftrightarrow\phi(x)$.
\item If $x \in \UF(P_{\mr seq})$ is coded as a decreasing sequence $\langle \sigma_n : n \in \mathbb{N} \rangle$ of $P$, 
 then there exists a $\mathbf{\Sigma}_{n}^0$-subset $U \subseteq \UF(P_{\mr seq})$ such that $x \in U \leftrightarrow\phi(x)$.
\end{enumerate}
\end{prop}

\begin{proof}
Since $\phi(\upcl(x))$ is a $\mathbf{\Sigma}_{n+1}^0$-formula for a decreasing sequence $x\in \UF(P_{\mr seq})$, the first case follows immediately from the second.
In the second case,  we proceed by induction on $n\in\omega$, presenting the case for $n=1$.
When $x$ is coded by a decreasing sequence $x=\langle\sigma_k:k\in\N\rangle$, the $\Sigma_1^0$-formula $\phi(x)$ is equivalent to $\exists m\,\theta(x[m])$ by Kleene's Normal Form Theorem (cf.~\cite[II.2.7]{SOSOA}), where $\theta(\tau)$ is a $\Sigma_0^0$-formula. Here, $x[m]= \langle \sigma_0, \dots, \sigma_{m-1} \rangle$ denotes the finite initial segment of $x$ of length $m$.
We define $U$ as the open set encoded by the collection of finite sequences $\tau\in P^{<\N}$ satisfying $\theta(\tau)$. 
The condition $\theta(x[m])$ is equivalent to saying that $x$ belongs to the basic open set $N_{x(m-1)}$. In this representation, the existential quantifier $\exists m$ corresponds exactly to taking a countable union of basic open sets. Therefore, $U$ is a $\mathbf{\Sigma}^{0}_{1}$-subset (an open set) satisfying $x \in U \leftrightarrow\phi(x)$.
\end{proof}
Even in the first case, if $x$ occurs only positively in $\phi(x)$ (meaning that $\phi$ relies on information of the form $p\in x$, but not $p\notin x$), we can still draw the same conclusion as in the second case. This proposition is related to Proposition \ref{qPfml}.

From this point forward, we represent filters as decreasing sequences rather than sets and we write $\mathbb{F}(P)$, $\UF(P)$, $\MF(P)$, and $\NP(P)$ for $\mathbb{F}(P_{\mr seq})$, $\UF(P_{\mr seq})$, $\MF(P_{\mr seq})$, and $\NP(P_{\mr seq})$, thereby avoiding the unnecessary invocation of $\ACAo$. Nevertheless, in contexts where  $\ACAo$ is already assumed, or where filters can be treated as sets within $\RCAo$ without risk of confusion, we may employ the set notation for the sake of simplicity.
Specifically, we say that a sequence $\langle q_i:i\in\N\rangle$ in $P$ is a strictly decreasing sequence if $q_{i+1}<_Pq_i\land q_{i}\nless_Pq_{i+1}$ for all $i\in\N$. Any such sequence generates a filter $F=\{p\in P:\exists i(q_i\leq_P p)\}$, which is necessarily non-principal. In a handy poset, every unbounded filter contains (and is generated by) such a strictly decreasing sequence.

\begin{rem}
The definition of handyness presented here may appear stronger than the one in Section 2. However, this modification is intended solely to adapt the theory to the subsystems of second-order arithmetic. In fact, $\UF(P)$ is always a $G_{\delta}$ subset of $\mathbb{F}(P)$ in the standard topological sense.

\end{rem}

\begin{prop}[Handyness in $\RCAo$; cf. Proposition \ref{1-typicalization}]\label{typicalization}

The following are provable in $\RCAo$:
\begin{enumerate}
 \item For every countable poset $P\subseteq\N$, there exists a handy poset $P^{\prime}$ such that $\UF(P)$ is homeomorphic to $\UF(P^{\prime})$.
 \item For every countable poset $P\subseteq\N$, there exists a handy poset $P^{\prime}$ such that $\mathbb{F}(P)$ is homeomorphic to $\UF(P^{\prime})$.
\end{enumerate}
\end{prop}

\begin{proof}
We reason within $\RCAo$. 
\begin{enumerate}
\item Let $P^{\prime}=\{(p,n)\in P\times\N :\forall q<_{\N}n(q\nless_P p)\}$, where $\leq_{\N}$ denotes the usual order on $\N$. The order on $P^{\prime}$ is defined by ${\leq_{P^{\prime}}}:=\{((p,n),(q,m))\in P^{\prime}\times P^{\prime}:(p\leq_P q\land n>_{\N} m)\lor(p= q\land n= m)\}$. These sets are definable by $\Sigma^{0}_{0}$-comprehension. This poset $P^{\prime}$ is handy. Indeed, for every element $(p, n)$, there exists a smaller element $(p, n)>_{P^{\prime}}(p, n+1)\lor(p, n)>_{P^{\prime}}(q, n+1)$ where $n<_{\N}q$, and $\forall n\exists p\exists i((p, n)\geq_{P'}a_i)$ if $F=\langle a_i\rangle_{i\in\N}$ is a non-principal filter. Moreover, the sequence of open sets $\langle U_{n}\rangle_{n}$, where $U_n:=\bigcup\{N_{s}:p\in P, s=(p,n)\in P^{\prime}\}$ satisfies $\bigcap_{n\in\N}U_{n}=\UF(P)$. We define a code for a continuous map $\Phi:\UF(P)\rightarrow \UF(P^{\prime})$ and its inverse $\Psi:\UF(P^{\prime})\rightarrow \UF(P)$ as follows in $\RCAo$:
\[
\Phi=\{(0,p,(p,n)):(p,n)\in P^{\prime},p\in P\}
\]
\[
\Psi=\{(0,(p,n),p):(p,n)\in P^{\prime},p\in P\}
\]
As in the proof of Proposition \ref{1-typicalization}.1, these codes define total continuous functions that are inverses of each other, establishing $\UF(P)\cong\UF(P^{\prime})$.

\item Let $P^{\prime}=\{(p,n)\in P\times\N :p\leq_{\N} n\}$ with the order ${\leq_{P^{\prime}}}:=\{((p,n),(q,m))\in P^{\prime}\times P^{\prime}:(p\leq_P q\land n>_{\N} m)\lor(p= q\land n= m)\}$. These sets are definable by $\Sigma^{0}_{0}$-comprehension. This $P^{\prime}$ is handy because any filter $F=\langle a_i\rangle_{i\in\N}$ in $P^{\prime}$ is unbounded if and only if $\forall n\in\N \,\exists p\in P\,\exists i\,(p,n)\geq_{P'} a_i$. We define the codes for the homeomorphisms $\Phi:\mathbb{F}(P)\rightarrow \UF(P^{\prime})$ and $\Psi:\UF(P^{\prime})\rightarrow \mathbb{F}(P)$ as follows:
\[
\Phi=\{(0,p,(p,n)):(p,n)\in P^{\prime},p\in P\}
\]
\[
\Psi=\{(0,(p,n),p):(p,n)\in P^{\prime},p\in P\}
\]
It can be verified within $\RCAo$ that $\Phi$ and $\Psi$ are total and mutually inverse, following the logic of Proposition \ref{1-typicalization}.2. \qedhere
\end{enumerate}
\end{proof}

\begin{prop}[cf.~Proposition \ref{prop:1-NPtoNPUF}]\label{NPtoNPUF}
Within $\RCAo$, the following hold for every countable poset $P$.
\begin{enumerate}
    \item There exists a countable poset $P^{\prime}$ such that $\NP(P)\cong\NP(P^{\prime})\cap\UF(P^{\prime})$.
    \item Conversely, there exists a countable poset $P^{\prime\prime}$ such that $\NP(P)\cap\UF(P)\cong\NP(P^{\prime\prime})$. 
\end{enumerate}
\end{prop}
\begin{proof}
\begin{enumerate}
    \item Let $P^{\prime}=\{p^{\prime}:p\in P\}$ be a poset equipped with the order $p^{\prime}\geq_{P^{\prime}}q^{\prime}\Leftrightarrow (p>_Pq\land p<_{\N}q)\lor(p=q)$, where $\leq_{\N}$ denotes the usual order on $\N$. Under this construction, every non-principal filter $F\in\NP(P)$ belongs to $\NP(P^{\prime})\cap\UF(P^{\prime})$. Their codes of continuous maps is as follows:
\begin{align*}
\Phi:=\{(n,p,q)\in\N\times P\times P':p=q\}\\
\Psi:=\{(n,q,p)\in\N\times P'\times P:q=p\}.
\end{align*}
\item We define the poset $P^{\prime\prime}$ as follows:
\[
P^{\prime\prime}=\{(n,p)\in\N\times P:\forall q\leq n(p\not\geq_P q)\},
\] with the order $(n,p)\geq_{P^{\prime\prime}}(m,q)\Leftrightarrow (p>_Pq\land m>_{\N}n)\lor (p=q\land n=m)$. For any $F\in\NP(P)\cap\UF(P)$, the mapping $f(F)=\{(n,p)\in P^{\prime\prime}:p\in F\}$ yields an element of $\NP(P^{\prime\prime})$. Conversely, for every $G\in\NP(P^{\prime\prime})$, the set $g(G)=\{p\in P:(0,p)\in G\}$ is in $\NP(P)\cap\UF(P)$.
Their codes of continuous maps is as follows:
\begin{align*}
\Phi:=\{(n,p,(m,q))\in\N\times P\times P'':p=q\}\\
\Psi:=\{(n,(m,q),p)\in\N\times P''\times P:q=p\}.
\end{align*}
The technical details of these proofs follow the same method as those discussed in Proposition \ref{typicalization}. \qedhere
\end{enumerate}
\end{proof}
Consequently, by appropriately reconstructing the underlying poset as shown above, we may safely assume that $\NP(P)\subseteq\UF(P)$ without loss of generality.

Within $\RCAo$, NP spaces are equivalent to ideal spaces.
\begin{defi}[Ideal spaces]
Let $\prec_P$ be a transitive relation on a subset $P \subseteq \N$. A subset $F \subseteq P$ is called an \textit{ideal as a set} if it satisfies the following conditions:
\begin{enumerate}
\renewcommand{\labelenumi}{(\roman{enumi})}
 \item$\forall p,p^{\prime}\in P((p\in F\land p^{\prime}\prec_P p)\rightarrow p^{\prime}\in F)$;
 \item
$\forall p,p^{\prime}\in P((p,p^{\prime}\in F)\rightarrow\exists r\in F(p,p^{\prime}\prec_P r))$;
 \item
$F\neq\emptyset$.
\end{enumerate}
An increasing sequence $F=\langle p_{i}: i\in\N\rangle$ of $P$ is called an \textit{ideal} if the set $\{p\in P:\exists q\in F(p\prec_P q\lor p = q)\}=\{p\in P:\exists i(p\prec_P p_{i}\lor p=p_i)\}$ forms an ideal as a set. Then, two ideals $F=\langle p_i:i\in\N\rangle$ and $F^{\prime}=\langle q_j:j\in\N\rangle$ are defined to be equal as ideals if and only if $\forall i\exists j (p_i\prec_P q_j)$ and $\forall j\exists i (q_j\prec_P p_i)$. 
We abbreviate the statement that a sequence $(a_n)_{n\in \N}$ is an ideal on $(P,\prec_P)$ as $(a_n)_{n\in \N}\in\mathbf{I}(P,\prec_P)$. 
A code for the open basis of $\mathbf{I}(P,\prec_P)$ is given by $P$. A point (i.e., an ideal) $x\in\mathbf{I}(P,\prec_P)$ belongs to the basic open set $N_p$, coded by $p\in P$ if $\exists i(p\prec_P p_{i}\lor p=p_i)$.
\end{defi}

We can formalize Proposition \ref{prop:1-NPtoNPUF}.3 within second-order arithmetic as well.
\begin{prop}\label{prop:NPtoIdeal}
Within $\RCAo$, every NP space is representable as an ideal space, and vice versa.
\end{prop}
\begin{proof}
Let $(P, \leq_P)$ be a preorder such that $X = \NP(P)$. By replacing $\leq_P$ with the partial order $\leq_{P^{\prime}}$ defined by $a \leq_{P^{\prime}} b \leftrightarrow a <_P b \land (b \nless_P a \lor a <_{\mathbb{N}} b)\lor(a=b)$, we can assume without loss of generality that $(P, \leq_P)$ is a partial order. Defining a transitive relation $\prec_P$ on $P$ by 
\[
p \prec_P q \leftrightarrow q <_P p,
\]
we immediately obtain a natural homeomorphism $\mathbf{I}(P,\prec_P)\cong \NP(P)$.

Conversely, let $\prec_P$ be a transitive relation on a subset $P \subseteq \N$. We define a partial order $\leq_Q$ on $Q := P \times \N$ by setting 
\[
(a,m) \leq_Q (b,n) \leftrightarrow (n < m \land b \prec_P a)\lor(a=b\land m=n).
\]
Then, the codes for continuous functions defined by
\[
\Phi := \{ (0, p, (p,n)) : (p,n) \in Q \} \quad \text{and} \quad \Psi := \{ (0, (p,n), p) : (p,n) \in Q \}
\]
yield a homeomorphism $\mathbf{I}(P, \prec_P) \cong \NP(Q)$.
\end{proof}
This proof is based on \cite[Proposition 1]{BKS2024}.

\begin{defi}[Universal UF space]

Let $\mc P_{\fin}(\N)$ be the collection of all finite subsets of $\N$, represented by the set of all strictly increasing sequences $\sigma\in\text{Seq}$ (i.e., $\forall n<\lh(\sigma)-1(\sigma(n)<\sigma(n+1))$). We view $\mc P_{\fin}(\N)$ as a poset ordered by reverse inclusion: $t\geq_P s\Leftrightarrow t\subseteq s$. 
We consider $\mc P(\N)$ as a quasi-Polish space equipped with the quasi-metric $d:\mc P(\N)\times\mc P(\N)\to[0,\infty)$ defined by:
\[
d(F,G)=
\begin{cases}
2^{-n}&\text{ if }n=\min\{k\in\N :k\in F\land k\notin G\}\\
0&\text{ if }F\subseteq G
\end{cases}
\]
where $F,G\subseteq \N$. Under this quasi-metric, $\mc P(\N)$ is homeomorphic to the space $\mathbb{F}(\mc P_{\fin}(\N))\cup\{\emptyset\}$.
\end{defi}
\begin{rem}
Since filters are defined as decreasing sequences of basic open sets, points in $\mc P(\N)$ cannot be simply denoted by their characteristic functions $\xi:\N\to\{0,1\}$. Such a $\xi$ is compatible with a filter as a set. Specifically, we denote a point $x\in\mc P(\N)$ as a sequence $x=\langle q_i\in\mc P_{\fin}(\N):i\in\N\rangle$ satisfying the following conditions for all $i<j$:
\begin{enumerate}
\item $q_i\subseteq\{0,\cdots,i\}$;
\item $q_i\subseteq q_j$.
\end{enumerate}
For a point $x = \langle q_i : i \in \mathbb{N} \rangle$ and $n \in \mathbb{N}$, we write $n \in x$ if $n \in q_i$ for some $i \in \mathbb{N}$.
This representation forms an effective left-Cauchy sequence with respect to the underlying quasi-metric. However, it is important to note that such a sequence is not generally an effective $\hat{d}$-Cauchy sequence. 
\end{rem}

From now on, we establish the equivalence between the representations of UF spaces and $\mathbf{\Pi}_2^0$ spaces, while also addressing NP spaces and NPUF spaces. The logical strengths of switching these representations fall into two distinct levels, $\RCAo$ and $\PCAo$.

\begin{prop}\label{UFtoPi20}
The following statement is provable in $\RCAo$:
Every UF space is homeomorphic to a $\mathbf{\Pi}_{2}^{0}$-subspace of $\mc P(\N)$.
\end{prop}
\begin{proof}
We reason within $\RCAo$.
Let $X$ be a UF space and $P\subseteq\N$ be a handy poset such that $X=\UF(P)$.
Define a mapping $g:P\rightarrow \mc P_{\fin}(\N)$ by $g(p)=\{q\in P:q\leq_{\N} p\land p\leq_P q\}$ with $\Sigma_{0}^{0}$-comprehension, where $\leq_{\N}$ denotes the usual order on $\N$. 

For an unbounded filter $F=\langle q_i:i\in\N\rangle\in\UF(P)$, we define its image under $\hat{g}$ as the sequence of $\mc P_{\fin}(\N)$:
\[
\hat{g}(F):=\langle g(q_i):i\in\N\rangle.
\]
This $\hat{g}(F)$ is definable in $\RCAo$. We verify that this $\hat{g}$ satisfies the required conditions:

 \begin{itemize}
  \item $\hat{g}(F)$ \textbf{is a filter} in $\mc P_{\fin}(\N)$: By the definition of $g$, $\hat{g}(F)$ is a decreasing sequence.

\item \textbf{The image of $\hat{g}$ is }$\mathbf{\Pi}_{2}^{0}$: We define a formula $\phi(n,\sigma)$ for $n\in\N$ and $\sigma\in P^{<\N}$ stating that the length of $\sigma$ is $n+1$, $\forall i<n(\sigma(i)>_P\sigma(i+1))$ and $\forall p<n+1(\sigma(n)\ngeq_P p)$.
 Let $U_{n}=\bigcup\{N_{g(\sigma(n))}:\phi(n,\sigma)\}$ be a sequence of open sets in $\mc P(\N)$. Furthermore, let $C\subseteq\mc P(\N)$ be the $\mathbf{\Pi}_{2}^{0}$ set consisting of points $x$ that satisfy the filter axioms:
\begin{align*}
C&=\bigcap_{l\notin P}\{x\in\mc P(\N):l\notin x\}\cap&&\bigcap_{p,q\in P}\{x\in\mc P(\N):(p\in x\land q\in x)\\
&&&\to\exists r(r\in x\land r\in P\land r\leq_P p\land r\leq_P q)\}\\
&=\bigcap_{l\notin P}(\mc P(\N)\backslash N_{l})\cap\bigcap_{p,q\in P}&&(\bigcup\{N_r:p,q\geq_P r\}\cup(\mc P(\N)\backslash N_{p,q})).
\end{align*} Let $B=\bigcap_{n\in\N} U_{n}\cap C$. Since $P$ is handy, $\hat{g}(\UF(P))\subseteq B$. Conversely, for any sequence $x=\langle q_i\rangle_{i\in\N}\in B$, the sequence $F_x:=\langle q_i\rangle_{i\in\N}$ is a decreasing sequence, which determines an unbounded filter on $P$. Indeed, 
if $p,q\in x$, $p,q\in P$ and there is $r\in x$ such that $p,q\geq_P r$ since $x\in C$. Thus $F_x$ is a filter on $P$. 
The condition $x\in\bigcap_{n} U_n$ ensures that, for each $n$, there is a finite sequence $\sigma$ with $\phi(n,\sigma)$ and $\exists i(\sigma(n)\geq_P q_i)$ with $F_x=\langle q_i\rangle_{i\in\N}$. Hence, for each $n$, $F_x$ contains an element which is not above any $p<_\N n+1$, and therefore $F_x$ is non-principal (and thus unbounded).

  \item $\hat{g}$ \textbf{is a bijection onto its image}: For any $F\in \UF(P)$, let $\hat{F}=\{p\in P:p\in\bigcup \hat{g}(F)\}$ in the set form. By the definition of $g$, $p\in\bigcup\hat{g}(F)\Leftrightarrow p\in F$, hence $F=\hat{F}$. This shows that $\hat{g}$ is injective and its image is exactly $B$.
  \item $\hat{g}$ \textbf{is a homeomorphism}: $\hat{g}$ is continuous because the preimage of a basic open set $N_q\subseteq\mc P(\N)$ (where $q\in\mc P_{\fin}(\N)$) is given by $\hat{g}^{-1}(N_{q})=\{F\in \UF(P):\forall p\in q\,(p\in F)\}$ which is an open set in $\UF(P)$. Moreover, $\hat{g}$ is an open mapping since $\hat{g}$ is induced by $g$.
The codes of continuous maps $\hat{g}$ and $\hat{g}^{-1}$ are as follows:
\begin{align*}
\Phi:=\{(n,r,q)\in\N\times P\times\mc P_{\fin}(\N): \forall p\in q(r\leq_P p)\}\\
\Psi:=\{(n,q,p)\in\N\times \mc P_{\fin}(\N) \times P: g(p)=q\}. 
\end{align*} \qedhere
\end{itemize}
\end{proof}

\begin{cor}\label{cor:NPUFtoPi02}
The following statement is provable in $\RCAo$:
For every countable poset $P$, the subspace $\NP\UF(P):=\NP(P)\cap\UF(P)$ is homeomorphic to a $\mathbf{\Pi}_{2}^{0}$-subspace of $\mc P(\N)$.
\end{cor}
\begin{proof}
By Proposition~\ref{typicalization}, there exists a handy poset $P'$ such that $\UF(P)\cong\UF(P')$. Then, by Proposition~\ref{UFtoPi20}, there is a $\mathbf{\Pi}_2^0$ subspace $X$ of $\mc P(\N)$ such that $\UF(P')\cong X$. 
Let $P=\langle p_i:i\in\N\rangle$ be an enumeration of $P$ with repetition. Furthermore, we can uniformly enumerate $Q_i:=\{q\in P:q<_Pp_i\land p_i\nless_P q\}\subseteq P$ as $\langle q_{i,n}:n\in\N\rangle$ with repetition. 
Then $\NP\UF(P)$ is a $\mathbf{\Pi}_2^0$ subspace of $\UF(P)$ coded by $(0,2,\langle (0,p_i),(1,q_{i,n}):i,n\in\N\rangle)$ (in fact, it forms a closed subset). From this code, one can construct the code for the corresponding $\mathbf{\Pi}_2^0$ subspace of $X$ within $\RCAo$.
\end{proof}

\begin{prop}\label{P(N)-to-UF(P)}
Within $\PCAo$, every $\mathbf{\Pi}_{2}^{0}$-subspace of $\mc P(\N)$ is homeomorphic to a UF space over a handy poset.
\end{prop}
\begin{proof}
We reason within $\PCAo$.
This proof is a refinement of the previous proposition to ensure the overtness of the poset.
Let $X\subseteq\mc P(\N)$ be a non-empty $\mathbf{\Pi}_2^0$-subspace $\bigcap_{n\in\N}(A_n\cup B_n)$ of $\mc P(\N)$ where each $A_n$ is a closed set and $B_n$ is an open set. We first define a preliminary poset $P^{\prime}$ by: 
\[
P^{\prime}:=\{(n,q)\in\N\times\mc P_{\fin}(\N):q\in \bigcap_{i=0}^{n}A_i\cup B_i\}.
\]
The order on $P^{\prime}$ is defined as $(n,q)\leq_{P^{\prime}}(m,r)\Leftrightarrow (n>m\land q\supseteq r)\lor(n=m\land q=r)$ which is available via $\Sigma_{0}^{0}$-comprehension. By $\PCAo$, we can define a subposet $P\subseteq P^{\prime}$ consisting of elements from which an infinite descending path originates. The defining condition is $\Sigma^1_1$, and hence the set exists in $\PCAo$:
\begin{align*}
P=\{(n,q)\in P^{\prime}:&\exists h:\N\rightarrow\mc P_{\fin}(\N)\\
&(\forall m\in\N (h(m)\in\bigcap_{i=0}^{m}A_i\cup B_i\land h(m)\subseteq h(m+1))\land h(n)=q)\}.
\end{align*}
By construction, $P$ is a handy poset, as every element in $P$ has a successor in $P$ by the definition of infinite paths.
We define the following mappings to establish the homeomorphism:
\begin{itemize}
\item $f:\UF(P)\rightarrow X$ is defined by $f(F)=\{i\in\N:\exists (n,q)\in F (i\in q)\}=\{i\in\N:p_{0,\{i\}}\in F\}$ in the set form. The code $\Phi$ for $f$ is given by $\Phi:=\{(m,(n,q),r)\in \N\times P\times\mc P_{\fin}(\N):r\subseteq q\}$. $\Sigma_{0}^{0}$-comprehension in $\RCAo$ suffices to ensure $f$ is a well-defined continuous map.
\item $g:X\rightarrow \UF(P)$ is defined by $g(G)=\upcl(\{(n,q)\in P:q\subseteq G\})$ in the set form. The code $\Psi$ for $g$ is given by $\Psi:=\{(m,r,(n,q))\in \N\times\mc P_{\fin}(\N)\times P:q\subseteq r\}$. For any $G\in X$, since $G\in\bigcap_{n\in\N}(A_n\cup B_n)$ for all $n$, the collection of its finite approximations forms an unbounded descending sequence in $P$.
\end{itemize}
It is straightforward to verify that $f$ and $g$ are continuous and are inverses of each other. This completes the proof.
\end{proof}

Working with the intersection of NP and UF spaces allows this proposition to be proved even within the framework of $\RCAo$.

\begin{prop}
Within $\RCAo$, every $\mathbf{\Pi}_{2}^{0}$-subspace of $\mc P(\N)$ is homeomorphic to $\NP\UF(P):=\NP(P)\cap\UF(P)$, for some countable poset $P$.
\end{prop}
\begin{proof}
We reason within $\RCAo$.
Let $X\subseteq\mc P(\N)$ be a nonempty $\mathbf{\Pi}_2^0$-subspace defined by $X=\bigcap_{n\in\N}(A_n\cup B_n)$, where each $A_n$ is a closed set coded as $(0,1,\langle q_{n,k}:k\in\N\rangle)$ and $B_n$ is an open set coded as $(1,1,\langle  r_{n,h}:h\in\N\rangle)$ for all $n\in\N$. Define a poset $P^{\prime}$ as follows:
\[
P^{\prime}=\{(i,l,q)\subseteq\N\times\N\times\mc P_{\fin}(\N):\forall(n,k)<i\,\exists h<l\,(q_{n,k}\subseteq q\rightarrow r_{n,h}\subseteq q)\}.
\] The order on $P^{\prime}$ is defined by $(i,l,q)\geq_{P^{\prime}}(j,k,r)\Leftrightarrow (l>k\land q\supseteq r\land i\geq j)\lor(i=j\land l=k\land q=r)$. This set $P^{\prime}$ exists by $\Sigma_{0}^{0}$-comprehension. We define two mappings to establish the homeomorphism:
\begin{enumerate}
\item $f:\NP\UF(P^{\prime})\rightarrow X$ is defined by $f(F)=\langle \{j\in\N:\exists (i,l,q)\,\exists k<m (p^{\prime}_k=(i,l,q)\land j\in q)\}:m\in\N\rangle$ for all $F=\langle p^{\prime}_m:m\in\N\rangle$. In the set form, 
\[
f(F)=\{j\in\N:\exists (i,l,q)\in F (j\in q)\}=\{i\in\N:(0,0,\{j\})\in F\}.
\]
The code $\Phi$ for $f$ is given by $\Phi:=\{(n,(i,l,q),r)\in\N\times P'\times\mc P_{\fin}(\N):r\subseteq q\}$.
For any filter $F\in\NP\UF(P)$, the non-principality ensures that the second components $l$ of elements $(i,l,q)$ are arbitrarily large. Thus, the union $f(F)=\bigcup \{q\in\mc P_{\fin}(\N):\exists (i,l,q)\in F\}$ represents a unique element $G\in\mc P(\N)$ that satisfies $G\in\bigcap_{n\in\N}(A_n\cup B_n)=X$.
\item $g:X\rightarrow \NP\UF(P^{\prime})$ is defined by $g(G)=\langle(i,l_i,q_i)\in P^{\prime}:i\in\N\rangle$ such that $(\forall m<i \exists m^{\prime}(r_m\subseteq q_i\subseteq r_{m^{\prime}}))$ for all $G=\langle r_m:m\in\N\rangle$.
Since $G$ is in $X$, there is a $m^{\prime},l_i,q_i$ such that $(i,l_i,q_i)\in P^{\prime}$ and $q_i\subseteq r_{m^{\prime}}$. 
In the set form, 
\[
g(G)=\{(i,l_i,q_i)\in P^{\prime}:G\cap\{0,\cdots, i\} \subseteq q_i\subseteq G\}.
\] The code $\Psi$ for $g$ is given by $\Psi:=\{(n,(i,l,q),r)\in\N\times\mc P(\N)\times P':q\subseteq r\}$.
 This ensures that $g(G)$ is both a non-principal and an unbounded filter.
\end{enumerate}
The existence of $f$ and $g$ as continuous codes is guaranteed by $\RCAo$, as $g(G)$ is  a descending sequence generated by $G$. Thus, $X\cong \NP\UF(P^{\prime})$.
\end{proof}
The above proposition follows immediately from Propositions \ref{NPtoNPUF} and \ref{prop:NPtoIdeal}, together with \cite[Theorem 12]{de2020overt}.

By taking a pruned subposet of $P$, we can show that $\NP\UF(P)$ is itself a UF space.  
\begin{prop}[cf.~\ref{1-typicalization}.3]\label{NPUFtoUF}
For every countable poset $P$, the subspace $\NP\UF(P):=\NP(P)\cap\UF(P)$ is homeomorphic to a UF space over a subposet $P^{\prime}\subseteq P$ within $\PCAo$.
\end{prop}
\begin{proof}
We reason within $\PCAo$.
By Proposition~\ref{NPtoNPUF}.2, we can define a poset $Q \subseteq \mathbb{N} \times P$ such that $Q$ satisfies $\NP\UF(P) \cong \NP(Q)$ and $\NP(Q) \subseteq \UF(Q)$.
Let $P^{\prime}\subseteq Q$ be the subposet defined by: 
\[
p\in P^{\prime}\leftrightarrow(\exists f:\N\rightarrow Q)(\forall n\in\N(f(n)>_Qf(n+1)\land f(n)\nleq_Qf(n+1))\land f(0)=p)
\] 
Since the existence of an infinite descending sequence is a $\Sigma_1^1(Q)$ condition, the set $P^{\prime}$ exists by $\Sigma^1_1$-comprehension and has the required property. By construction, every element in $P^{\prime}$ is the starting point of at least one non-principal (and thus unbounded) filter. Therefore, the poset space over $P^{\prime}$ represents the intended intersection of non-principal and unbounded filters, satisfying the proposition.
\end{proof}
The proof above is a second-order arithmetic version of the proof of Proposition~\ref{1-typicalization}.3. Note that we can see this result by combining Proposition~\ref{P(N)-to-UF(P)} and Corollary~\ref{cor:NPUFtoPi02}.

\begin{cor}
Within $\PCAo$, for every countable poset $P$, the non-principal space $\NP(P)$ is homeomorphic to a UF space over some handy poset $Q$.
\end{cor}
\begin{proof}
The result follows directly from Proposition \ref{NPtoNPUF} and Proposition \ref{NPUFtoUF}. 
\end{proof}

Conversely, we shall now show that the existence of such homeomorphisms (or the properties discussed above) implies $\PCAo$ over $\RCAo$. 
The axiom system $\PCAo$ is employed to determine whether a basic open set is empty. This highlights the necessity of the overtness assumption, as discussed in \cite{BKS2024}.

\begin{prop}\label{equi-overt-PCAo}
The following statements are equivalent over $\RCAo$:
\begin{enumerate}
\item $\PCAo$.
\item Every $\mathbf{\Pi}_{2}^{0}$-subspace $Y$ of  $\mc P(\N)$ is homeomorphic to a UF space. Specifically, there exists a countable poset $P$ and codes for total continuous functions $f:Y\rightarrow \UF(P)$, $g:\UF(P)\rightarrow Y$ such that $f(g(F))=F$ and $g(f(y))=y$ for all $F\in \UF(P),y\in Y$. 
\end{enumerate}
\end{prop}
\begin{proof}
While the implication $(1)\Rightarrow(2)$ has already been established, we show the converse $(2)\Rightarrow(1)$.

We first show that $(2)\Rightarrow(1)$ implies $\ACAo$ over $\RCAo$.
Let $\phi:\N\to\N$ be a given injective function. To prove $\ACAo$, it is sufficient to show that the range of $\phi$, $\mathrm{rng}(\phi)$, exists as a set. We define a poset $P$ as follows:
\[
(n,m)\in P\Leftrightarrow \forall i<n(\phi(i)\neq m)
\]
equipped with the ordering $(n,m) \geq_P (n',m')$ if and only if $n \leq n'$ and $m=m'$. By Proposition \ref{NPtoNPUF}, there exists a countable poset $P^{\prime}$ such that $\NP(P)\cong\NP(P^{\prime})\cap\UF(P^{\prime})$. Since $\NP(P^{\prime})\cap\UF(P^{\prime})$ can be embedded as a $\mathbf{\Pi}_2^0$ subset of $\mc P(\N)$, it follows that $\NP(P)$ is homeomorphic to a $\mathbf{\Pi}_2^0$ subspace of $\mc P(\N)$.  Assuming $(2)$, there exists a countable handy poset $Q$ and a homeomorphism $\psi:\NP(P)\cong\UF(Q)$. Let $\Psi\subseteq\N\times P\times Q$ be the code for the continuous function $\psi$. We can then define the set $X$ using the information provided by the homeomorphism: 
\[
X:=\{m\in\N:\exists k\exists n\exists q\in Q( (k,(n,m),q)\in\Psi)\}=\{m\in\N:\forall i(\phi(i)\neq m)\}.
\]
Since $X=\N\setminus\mathrm{rng}(\phi)$ is defined via a formula available in this context, $\Delta_1^0$-comprehension ensures the existence of $\mathrm{rng}(\phi)$. Thus, the equivalence $(2)\Rightarrow(1)$ implies $\ACAo$.
To show $\PCAo$ over $\ACAo$, it suffices to show that for every sequence of trees $\langle T_n:n\in\N\rangle$ where each $T_n\in \N^{<\N}$, the set $X:=\{n\in\N:T_n\text{ has an infinite path}\}$ exists, assuming statement $(2)$. 
Given $\langle T_n:n\in\N\rangle$, let $T$ be a tree defined by $\forall n(\sigma\in T_n\leftrightarrow n\text{\textasciicircum}\sigma\in T)$ via $\Sigma_0^0$-comprehension. We construct a $\mathbf{\Pi}_2^0$-subspace $Y$ of $\mc P(\N)$ that encodes the paths of $T$. Let $N_{(i,j)}$ denote the basic open set in $\mc P(\N)$ coded by $\{(0,(i,j))\}=\{(0,k):k=(i,j)\}$, and let $M_{(i,j)}$ be its complement. We define closed sets $D_{\sigma}$ and open sets $E_j$ as follows:
\[
D_{\sigma}=
\begin{cases}
\mc P(\N)&\text{ if }\sigma\in T\\
\bigcup\{M_{(i,j)}:j<\lh(\sigma)\land\sigma(j)=i\}&\text{ if }\sigma\notin T\\
\end{cases}
\]
for each $\sigma\in\N^{<\N}$.
\[
E_j=\bigcup_{i\in\N}N_{(i,j)}
\]
for each $j\in\N$. Let $Y:=\bigcap_{\sigma\in\N^{<\N}}D_{\sigma}\cap\bigcap_{j\in\N}E_j$. Note that $Y$ is a $\mathbf{\Pi}_2^0$ set.
We claim that $T_n$ has an infinite path if and only if $Y\cap N_{(n,0)}$ is nonempty. If $x$ is an infinite path of $T_n$, then $F_n:=\{(n,0)\}\cup\{(i,j+1):x(j)=i\}$ is an element of $Y\cap N_{(n,0)}$. Conversely, if there exists $F\in Y\cap N_{(n,0)}$, we can define a sequence $y:\N\rightarrow\N$ by $y(0)=n$ and $y(j)=i$ the least $i$ such that $(i,j)\in F$ for $j>0$. The condition $\forall j(F\in E_j)$ ensures $y(j)$ is always defined, and $\forall\sigma(F\in D_{\sigma})$ ensures that every initial segment of $y$ is in $T$. Thus $y$ is an infinite path of $T$ starting with $y(0)=n$. So $y^{\prime}(j)=y(j+1)$ is an infinite path of $T_n$. The claim is proved. 
By statement $(2)$, there exists a poset $P$ and a homeomorphism $f:\UF(P)\rightarrow Y$. Since $f$ determined by a continuous code, the pre-image $f^{-1}(N_{(n,0)})$ is an open set in $\UF(P)$, which can be represented as a union of basic open sets $\bigcup_{p\in I_n}N_p$ for some $\Sigma_1^0(f)$ set of indices $I_n\subseteq P$.
Since $\UF(P) \cap N_p \neq \emptyset$ within $\RCAo$, the existence of a point in $Y \cap N_{(n,0)}$ can be determined by checking whether there exists $p \in f^{-1}(N_{(n,0)})$. Thus, we can define:
\[
X:=\{n\in\N:\exists p\in P(p\in f^{-1}(N_{(n,0)}))\}=\{n\in\N:T_n\text{ has an infinite path}\}
\] by arithmetical comprehension. Since the existence of $X$ for any sequence of trees implies $\PCAo$, and the proof is complete.
\end{proof}


\begin{cor}
The following statements are equivalent over $\RCAo$:
\begin{enumerate}
\item $\PCAo$.
\item For every countable poset $P$, there exists a poset $P^{\prime}$ such that the intersection $\UF(P)\cap\NP(P)$ is homeomorphic to $\UF(P^{\prime})$.
\item For every countable poset $P$, there exists a poset $P^{\prime}$ such that the space $\NP(P)$ is homeomorphic to $\UF(P^{\prime})$.
\end{enumerate}
\end{cor}

\begin{cor}
The following statements are equivalent over $\RCAo$:
\begin{enumerate}
\item $\PCAo$.
\item Every $\mathbf{\Pi}_2^0$-subspace of a UF space is itself homeomorphic to a UF space.
\end{enumerate}
\end{cor}
\begin{proof}
Let $X$ be a UF space and $Y\subseteq X$ be a $\mathbf{\Pi}_2^0$-subspace. By Proposition \ref{UFtoPi20}, $X$ is homeomorphic to a $\mathbf{\Pi}_2^0$-subspace of $\mc P(\N)$. Since a $\mathbf{\Pi}_2^0$-subspace of a $\mathbf{\Pi}_2^0$-subspace is again a $\mathbf{\Pi}_2^0$-subspace, $Y$ is homeomorphic to a $\mathbf{\Pi}_2^0$-subspace of $\mc P(\N)$. By Proposition \ref{P(N)-to-UF(P)}, any such subspace is homeomorphic to a UF space within $\PCAo$.
\end{proof}

\subsection{Quasi-Polish spaces}
Next, we define quasi-Polish spaces within the framework of second-order arithmetic, which yields a natural generalization of the classical definition of Polish spaces \cite{SOSOA}. Moreover, by taking the finite subsets $\mc P_{\fin}(\N)$ as a countable dense subset, the universal space $\mc P(\N)$ itself is definable as a quasi-Polish space. Consequently, this framework also subsumes the standard construction of the universal space.

\begin{defi}[quasi-Polish spaces]
The following definitions are formulated within $\RCAo$. A \textit{code for a second-countable quasi-metric space} $X_A$ is a tuple $(A,d)$ where $A\subseteq\N$ is a nonempty set and $d:A\times A\rightarrow\R$ is a function satisfying:
\begin{enumerate}
\item[(i)] $\forall a,b\in A\,(d(a,b)\geq0)$;
\item[(ii)] $\forall a,b,c\in A\,(d(a,b)+d(b,c)\geq d(a,c))$.
\end{enumerate}
In other words, $d$ is a pseudo-quasi-metric on $A$. 
A point of $X_A$ is a sequence $x=\langle a_n:n\in\N\rangle$ of elements in $A$ that is \textit{$\hat{d}$-Cauchy} with respect to the symmetrized pseudo-metric $\hat{d}(a,b)=\max\{d(a,b),d(b,a)\}$; specifically, $\forall m,n\, (m>n\rightarrow \hat{d}(a_n,a_m)<2^{-n})$. We write $x\in X_A$ as an abbreviation for ``$x$ is a point of $X_A$''.
For $x=\langle a_n:n\in\N\rangle,y=\langle b_n:n\in\N\rangle\in X_A$, $d(x,y)$ is defined as a real number $\langle d(a_{n+1},b_{n+1})\rangle_n$. 
 
A sequence $\langle a_n:n\in\N\rangle$ in $A$ is an \textit{effective left-Cauchy sequence} if $\forall n>0\,\forall m>n\,(d(a_{n},a_{m})<2^{-n})$. 
We say $X_A$ is a \textit{Smyth-complete quasi-metric space} if every effective left-Cauchy sequence $\langle a_n:n\in\N\rangle$ in $A$ has a point $x=\langle b_n:n\in\N\rangle$ such that $\forall n\, d(a_{n+1},b_{n+1})<2^{-n}$ and $\lim_n d(b_{n},a_{n})=0$. Two points $x$ and $y$ in $X_A$ are considered equivalent as a point ($x=_{X_A} y$) if $d(x,y)=d(y,x)=0$.


A code for an open subset $U$ of $X_A$ is a set of triples $U\subseteq\N\times A\times\Q_+$.   
A point $y\in X_A$ is an element of $U$ (denoted by $y\in U$) if there exist $n,a,r$ with $(n,a,r)\in U$ such that $d(a,y)<r$. 
In particular, we write $B_d(a,r)$ for the open set represented by the singleton code $\{(0,a,r)\}$, which corresponds to the basic open set coded by $(a,r)$ within the fixed basis code $A\times\Q_+$. Then we define Borel sets of finite ranks as in Definition \ref{defi:Borel-sets}.
Similarly, for a point $x=\langle a_n:n\in\N\rangle\in X_A$ and $r\in\Q_+$, $B_d(x,r)$ represents the set of triples $\{(n,a,l)\in\N\times A\times\Q_+:d(a_{n},a)+l+2^{-n}<r\}$ where $\langle a_n:n\in \N\rangle$ is a enumeration of $A$. 

Moreover, we write $B_{\hat{d}}(x,r)$ for the class of points $y\in X_A$ satisfying both $d(y,x)<r$ and $d(x,y)<r$. Note that $B_{\hat{d}}(x,r)$ is not an open set in general; it is a $\mathbf{\Sigma}^{0}_{2}$-set (Proposition~\ref{prop:d-hat-ball}).
\end{defi}
There is an alternative formalization of quasi-Polish spaces where points are represented as effective left-Cauchy sequences instead of effective $\hat{d}$-Cauchy sequences.
\begin{rem}
A quasi-metric space $(X,d)$ is called \textit{bicomplete} if the symmetrized metric space $(X,\hat{d})$ is a complete metric space. Hence, in the above definition, we gave a code $X_A$ for a quasi-metric space which is second-countable and bicomplete. However, be aware that a bicomplete quasi-metric space may not be Smyth-complete.

On the other hand, if $(X,d)$ is a second-countable Smyth-complete quasi-metric space, then there exists a countable set $D$ which is dense as a metric space $(X,\hat{d})$ by Proposition~\ref{prop:hat-d-dense}.
Moreover, it is straightforward to see that every Smyth-complete quasi-metric space is bicomplete.
Hence, any quasi-Polish space can be described by the above definition.
\end{rem}

\begin{prop}[{\cite[Proposition 12]{deBr2013}}, {\cite{kunzi1983strongly}}]\label{prop:d-hat-ball}
Let $X_A$ be a quasi-metric space with a code $(A,d)$. Every $\hat{d}$-open ball in $X_A$ is a $\mathbf{\Sigma}_2^0$-subset of $X_A$ with respect to the quasi-metric topology. 
\end{prop}
\begin{proof}
Let $B_{\hat{d}}(a,r)$ be a $\hat{d}$-open ball where $a\in A, r\in\Q_+$. 
\begin{align*}
B_{\hat{d}}(a,r)=B_{d}(a,r)\cap\bigcup\{X_A\setminus B_{d}(b,s):b\in A,s\in\Q_+,d(b,a)-s\geq r\}.
\end{align*}
\end{proof}

This proof can be formalized within $\RCAo$. Consequently, this proposition implies the following one.

\begin{prop}\label{qPfml}
Let $X_A$ be a quasi-metric space with a code $(A,d)$. Let $\phi(x)$ be a $\Sigma_n^0$-formula for $n\in\omega$ such that $x=y\land\phi(x)$ implies $\phi(y)$.
Within $\RCAo$, the following holds:
\begin{enumerate}
\item There exists a $\mathbf{\Sigma}^{0}_{n+1}$-subset $U\subseteq X_A$ such that $(x\in U)\leftrightarrow\phi(x)$ for all $x\in X_A$.
\item If $X_A$ is Smyth-complete and $x=y$ as an effective left-Cauchy sequence (not necessarily effective $\hat{d}$-Cauchy) and $\phi(x)$ implies $\phi(y)$, then for any $\Sigma^{0}_{n}$-formula $\phi(x)$, there exists a $\mathbf{\Sigma}^{0}_{n}$-subset $U\subseteq X_A$ such that $(x\in U)\leftrightarrow\phi(x)$.
\end{enumerate}
(cf. \cite[Lem II.5.7]{SOSOA}, Proposition \ref{UFfml})
\end{prop}

\begin{proof}
We proceed by induction on $n\in\omega$. Here, we present the proof for the base case $n=1$.
Let $x=\langle a_n:n\in\N\rangle$ be in $X_A$. By Kleene's Normal Form Theorem  \cite[ II.2.7]{SOSOA}, any $\Sigma_1^0$-formula $\phi(x)$ can be written in the form $\exists m\theta(x[m])$ where $\theta(\sigma)$ is a $\Sigma_0^0$-formula and $x[m]=\langle a_0,\cdots a_{m-1}\rangle$ is the initial segment of $x$. We define $U$ to be the $\hat{d}$-open set encoded by the collection of triples $(n,a,r)$ such that there exists a finite sequence $\sigma=\langle a_0,\cdots a_{m-1}\rangle$ satisfying $\theta(\sigma)$, with $a=a_m,r=2^{-m-1}$ and $(\forall i\leq m)(\forall j\leq m)(i<j\rightarrow \hat{d}(a_i,a_j)_{m+1}\leq 2^{-i-1})$. As established in the previous proposition, every $\hat{d}$-open ball is a $\Sigma_2^0$-subset. Since $(X_A,\hat{d})$ is a separable complete metric space, the open set $U$ as defined above satisfies $(x\in U)\leftrightarrow\phi(x)$ according to \cite[II.5.7]{SOSOA}. Thus, $U$ is a $\mathbf{\Sigma}_2^0$-subset of $X_A$ that satisfies $(x\in U)\leftrightarrow\phi(x)$.

Now, suppose $X_A$ is a Smyth-complete quasi-metric space, and let $x$ be an effective left-Cauchy sequence (which is not necessarily effective $\hat{d}$-Cauchy). We define $U$ as a $d$-open set consisting of triples $(n,a,r)$ where there exists a finite sequence $\sigma=\langle a_0,\cdots a_{m-1}\rangle$ satisfying $\theta(\sigma)$, with $a=a_m,r=2^{-m-1}$ and $(\forall i\leq m)(\forall j\leq m)(i<j\rightarrow d(a_i,a_j)_{m+1}\leq 2^{-i-1})$. If $y=\langle b_k:k\in\N\rangle \in U$, then there exists an initial segment $\sigma=\langle a_0,\cdots a_{m-1}\rangle$ such that $\theta(\sigma)$ holds and $d(a_m,y)<2^{-m-1}$. By the Smyth-completeness of $X_A$ with respect to left-Cauchy sequences, the concatenated sequence $z=\sigma\string^\langle b_n:m\leq n\rangle$ is an effective left-Cauchy sequence that converges to a point in $X_A$ such that $z=y$ in the quasi-metric sense.
The converse follows similarly, completing the proof for $n=1$.
\end{proof}

We now proceed to prove the equivalence between quasi-Polish spaces, UF spaces, and $\mathbf{\Pi}_2^0$ subspaces of $\mc P(\N)$.

\begin{thm}\label{qPtoUF}
Within $\RCAo$,
every second-countable Smyth-complete quasi-metric space is homeomorphic to a UF space.
\end{thm}
\begin{proof}
Let $(X_A,d)$ be a (second-countable) Smyth-complete quasi-metric space with the code $(A,d)$.
We define a poset $P=A\times\Q_+$ with the order $(a,r)\leq_P(b,s)\Leftrightarrow (s-d(b,a)>r)\lor(a=b\land r=s)$. Note that $(b,s)$ is ``larger'' (less precise) than $(a,r)$. We define the mapping $\phi:X_A\rightarrow \UF(P)$ within $\RCAo$ as follows:
\[
\phi(x)=\langle (a_i,2^{-i+1}):i\in\N\rangle
\]
where $x=\langle a_n\rangle_{n\in\N}$ is a representative $\hat{d}$-Cauchy sequence in $X_A$. We show that $\phi$ is a homeomorphism from $X_A$ to $\UF(P)$.  
\begin{itemize}
 \item \textbf{Well-definedness}: Let $x=\langle a_n\rangle_n=_{X_A} \langle b_n\rangle_n$ in $X_A$. For any $i$, $\exists j(d(a_i,x)+d(x,b_j)+2^{-j+1}<2^{-i+1})$, since $d(x,b_j)$ converges to $0$. Thus, $(a_i,2^{-i+1})>_P(b_j,2^{-j+1})$, which implies $\phi(\langle a_n\rangle_n)\subseteq\phi(\langle b_n\rangle_n)$. The reverse inclusion follows by symmetry.
Furthermore, $\phi(x)$ is an unbounded filter. The filter property follows from the Cauchy property of $\langle a_i\rangle_i$. To see unboundedness, note that for any $(b,s)\in P$ there is $i$ such that $(b,s)\not<_P (a_i,2^{-i+1})$.
 \item \textbf{Continuity}: A code for $\phi$ as a continuous function is given by the set $U:=\{(n,a,r,(b,s))\in\N\times A\times\Q_+\times P:(a,r)\leq_P(b,s)\}$. Within $\RCAo$, this defines $\phi$ as a continuous mapping from $X_A$ to $\UF(P)$.
 \item \textbf{Injectivity}: $\phi$ is injective. Suppose $\phi(x)=\phi(y)$ for $x=\langle a_i\rangle_i$ and $y=\langle b_j\rangle_j$. For any $n$, there exists $m$ such that $( a_n,2^{-n+1})>_P( b_m,2^{-m+1})$, implying $d(a_n,b_m)<2^{-n+1}$. $d(a_n,b_m)$ converges to $0$ since $\forall l,k(d(a_l,b_k)\leq d(a_l,a_n)+d(a_n,b_m)+d(b_m,b_k))$ and similarly $d(b_n,a_m)$ converges to $0$. Thus, $\hat{d}(x,y)=0$, so $x=_{X_A}y$.

 \item \textbf{Surjectivity}: Let $F=\langle(a_i,r_i):i\in\N\rangle\in\UF(P)$. Within $\RCAo$, we construct a sequence $Y=\langle(b_i,s_i):i\in\N\rangle$ in $P$ satisfying the following:
 \begin{enumerate}
  \item $(b_0,s_0)=(a_0,r_0)$.
  \item $(b_{n+1},s_{n+1})=(a_k,r_k)$ for the least $k$ such that \\$(b_n,s_n),(a_{n+1},r_{n+1})\geq_P(a_k,r_k)\land r_k<2^{-n-1}$.
 \end{enumerate}
Since $s_n$ converges to $0$, $(b_i)_i$ is an effective left-Cauchy sequence in $A$. By Smyth-completeness, it converges to some point $x=\langle c_k\rangle_k\in X_A$. It remains to verify that $\phi(x)=F$, equivalently,
\[
\exists i\,((a,r)\geq_P(a_i,r_i)) \leftrightarrow\exists k\,((a,r)\geq_P(c_k,2^{-k+1})).
\]
If $\exists i\,((a,r)\geq_P(a_i,r_i))$, then $(a,r)\geq_P(a_i,r_i)\geq_P(b_i,s_i)$. Since $d(b_i,x)<s_i$, there exist $K\in\N$ and $l\in\Q_+$ such that $s_i-d(b_i,c_n)>l$ for all $n>K$. Choose $j>K$ such that $2^{-j+1}<l$. Then $(c_j,2^{-j+1})<_P(b_i,s_i)$, and hence $(a,r)\geq_P(c_j,2^{-j+1})$.
Conversely, suppose that $(a,r)\geq_P(c_k,2^{-k+1})$ for some $k$. Since $s_m\to0$ and $(b_m)_m$ converges to $x$, we can choose $m$ such that $s_m<2^{-k}$ and $(b_m,s_m)<_P(c_k,2^{-k+1})$. Since $(b_m,s_m)=(a_i,r_i)$ for some $i\in\N$, it follows that $\exists i\,((a,r)\geq_P(a_i,r_i))$.
 \item \textbf{Openness}: $\phi$ is an open mapping because the image of a basic open set $B_d(a,r)$ in $X_A$ corresponds to the basic open set $N_{(a,r)}=\{F\in\UF(P):(a,r)\in F\}$ in the Scott topology of $\UF(P)$. Thus, a code for $\phi^{-1}$ can be constructed within $\mathsf{RCA}_0$. \qedhere
 
\end{itemize} 
\end{proof}

\begin{cor}
Within $\RCAo$,
every second-countable Smyth-complete quasi-metric space is homeomorphic to a $\mathbf{\Pi}_2^0$ subspace of $\mc P(\N)$.
\end{cor}
\begin{proof}
The result follows directly from Theorem \ref{UFtoPi20} and Theorem \ref{qPtoUF}.
\end{proof}

To show the converse, we first recall the following theorems.
Within $\RCAo$, let $X$ be a separable complete metric space. A closed subset $Y\subseteq X$ is called \textit{separably closed} if it is the closure of a countable subset $S\subseteq X$ (i.e., $Y=\bar{S}$, the class of limits of Cauchy sequences from $S=\langle x_{i}\in X: i\in\N\rangle$).
\begin{thm}[{\cite[Thm 2.12]{brown1990notions}}]
Over $\RCAo$, the statement ``every closed subset of a separable complete metric space is a separably closed subset'' is equivalent to $\PCAo$.
\end{thm}

Since every closed and separably closed subset is homeomorphic to a separable complete metric space within $\RCAo$, it is natural to attempt a similar definition of ``separable'' subsets within the context of second-countable quasi-metric spaces. However, several fundamental differences between metric and quasi-metric structures must be addressed:
\begin{itemize}
\item \textbf{Incomplete Convergence}: In a bicomplete quasi-metric space, not every left-Cauchy sequence necessarily converges.
\item \textbf{Deficit of Closedness:} In a second-countable Smyth-complete quasi-metric space $X$, the set $Y\subseteq X$ consisting of the limits of left-Cauchy sequences from a countable subset $S\subseteq X$ is not necessarily a closed subset in the classical sense. In general, such a set corresponds to a $\mathbf{\Pi}_2^0$ subset.
\item \textbf{Failure of Reconstruction from Density}: Unlike the metric case, a dense subset of a quasi-metric space does not always allow the original space to be recovered by taking the set of all convergence points. For example, the singleton $\{\N\}$ is a dense subset of $\mc P(\N)$ under any compatible quasi-metric, yet it fails to represent the full complexity of the space.
\end{itemize} 

We generalize the notions of closed and separably closed subsets to the quasi-metric setting via the following definitions.

\begin{defi}[Separable]
The following is formulated within $\RCAo$. Let $(X_A,d)$ be a quasi-metric space. A subset $B\subseteq X_A$ is said to be \textit{separable} if there exists a sequence of points $\langle x_i:i\in\N\rangle$ in $B$ such that every $x\in B$ is the limit of some subsequence consisting of elements from $\langle x_i:i\in\N\rangle$. Specifically, for each $x\in B$, there exists a sequence $\langle y_n\rangle\subseteq\{x_i:i\in\N\}$ such that $\text{lim}_{n\to\infty}\hat{d}(y_n,x)=0$. Note that $B$ is not required to be $\mathbf{\Pi}_2^0$; we only assume it is a definable subset.
\end{defi}

Furthermore, we introduce Smyth-separable and $\hat{d}$-separably closed subsets as generalizations of the metric notion of a ``completion of a countable subset''.
 
\begin{defi}[Smyth-separable and $\hat{d}$-separably closed]
 Let $(X,d)$ be a second-countable quasi-metric space. The following definitions are formulated within $\RCAo$:
\begin{enumerate}
\item A \textit{Smyth-separable subset} $S$ is coded by a sequence of points $\langle s_n:n\in\N\rangle$ in $X$. A point $x\in X$ belongs to $S$ (denoted by $x\in S$) if: 
\begin{align*}
\exists(x_n)_{n\in\N}&(\forall n\,\exists k\,(x_n=s_k)\land\langle x_n\rangle_n\text{ is effective left-Cauchy }\\
&\land\forall n(d(x_n,x)<2^{-n+1})\land \text{lim}_{n}d(x,x_n)=0).
\end{align*} If $(X,d)$ is Smyth-complete, the condition reduces to the existence of an effective left-Cauchy subsequence $\langle x_n\rangle_{n\in\N}$ such that 
\[
\forall n\exists k(x_n=s_k)\land\forall n(d(x_n,x)<2^{-n+1}).
\] 
\item A \textit{$\hat{d}$-separably closed subset} $C$ is coded by a sequence of points $\langle c_n : n \in \mathbb{N} \rangle$ in $X$. A point $x \in X$ belongs to $C$ if for every $r \in \mathbb{Q}_+$, there exists some $n$ such that $\hat{d}(x, c_n) < r$.
\end{enumerate}
Within $\RCAo$, every $\hat{d}$-separably closed set is a Smyth-separable subset.
Over $\ACAo$, a set is a Smyth-separable subset if and only if it is a $\hat{d}$-separably closed subset.  
Furthermore, every Smyth-separable subset ($\hat{d}$-separably closed subset) is a $\mathbf{\Pi}_2^0$-subset of $(X,d)$, since every $\hat{d}$-open ball is a $\mathbf{\Sigma}_2^0$-subset.
\end{defi}

We then have this theorem.

\begin{thm}\label{Pi20toqP}
Within $\PCAo$, every $\mathbf{\Pi}_2^0$ subset of $\mc P(\N)$ is homeomorphic to a second-countable Smyth-complete quasi-metric space.
\end{thm}

The proof follows directly from the two propositions presented below.

\begin{prop}\label{analyseparable}
Within $\PCAo$, every analytic subset of a quasi-metric space is separable.
This is a generalization of the theorem for closed subsets of metric spaces \cite[Thm 2.10]{brown1990notions}.
\end{prop}
\begin{proof}
Let $(X_A,d)$ be a quasi-metric space and $K\subseteq X_A$ be an analytic subset. We define a set $B$ using $\Pi_1^1$-comprehension as follows:
\[
B:=\{(b,s)\in A\times\Q_+:\exists x\in X_A(\hat{d}(b,x)<s\land x\in K)\}.
\]
Here, $\hat{d}$ denotes the symmetrized metric. Since $B$ is a countable set, we can enumerate its elements as $\langle(b_n,s_n):n\in\N\rangle$. 
For each $n\in\N$, the condition $\exists x\in X_A(\hat{d}(b,x)<s\land x\in K)$ is a $\Sigma_1^1$ formula. By $\Sigma_1^1$-$\mathsf{AC}$ (which is a consequence of $\PCAo$), there exists a sequence of points $\langle x_n:n\in\N\rangle$ in $X_A$ such that for each $n$, $\hat{d}(b_n,x_n)<s_n$ and $x_n\in K$.
$\{x_n:n\in\N\}$ is dense in $K$, proving that $K$ is separable.
\end{proof}
\begin{prop}\label{sepPi20qP}
The following is provable within $\RCAo$. Let $(X,d)$ be a second-countable Smyth-complete quasi-metric space. Suppose $Y\subseteq X$ is a separable $\mathbf{\Pi}_2^0$ subset that can be represented as an intersection of the form $Y=\bigcap_{i\in\N}(A_i\cup B_i)$ where for each $i\in\N$, $A_i$ and $B_i$ are disjoint.
Then, there exists a quasi-metric $d^{\prime}$ on $Y$ such that $(Y,d^{\prime})$ is homeomorphic to $(Y,d|Y)$ and $(Y,d^{\prime})$ is a second-countable Smyth-complete quasi-metric space.

\end{prop}
\begin{proof}
The argument proceeds identically to the proof of Proposition \ref{1-pi2-quasipolish} in Section 2. By applying the same estimation techniques for the reciprocal distance terms, we conclude that $(Y,d^{\prime})$ is a Smyth-complete quasi-metric space.
\end{proof}

We now proceed to derive $\PCAo$ from the assertion that these spaces are homeomorphic to quasi-Polish spaces.

\begin{cor}
Over $\RCAo$, the statement ``every closed subset of a second-countable Smyth-complete quasi-metric space is a $\hat{d}$-separably closed subset'' is equivalent to $\PCAo$.
\end{cor}
\begin{proof}
It is known that the statement ``every closed subset of a separable complete metric space is a separably closed subset'' is equivalent to $\PCAo$ over $\RCAo$ (cf. Brown \cite[Thm 2.12]{brown1990notions}). Since any separable complete metric space is naturally a second-countable Smyth-complete quasi-metric space (where $d=\hat{d}$), the claim for the quasi-metric case directly implies $\PCAo$.
Conversely, we show the claim holds within $\PCAo$. Let $X$ be a second-countable Smyth-complete quasi-metric space with code $(A,d)$, and let $Y$ be the associated separable complete metric space with code $(A,\hat{d})$. Any $d$-closed subset of $X$ is also a $\hat{d}$-closed subset of $Y$, as the $d$-topology is coarser than the $\hat{d}$-topology.
Under $\PCAo$, every $\hat{d}$-closed subset of $Y$ possesses a code $\langle s_n:n\in\N\rangle$ for a $\hat{d}$-separably closed subset. Since $\hat{d}$-convergence implies $d$-convergence, this sequence $\langle s_n\rangle$ also serves as a code for a $\hat{d}$-separably closed subset in $X$. Thus, the claim is provable within $\PCAo$.
\end{proof} 

By restricting our focus to $\mc P(\N)$, we can say the following.
\begin{cor}
Over $\RCAo$, the statement ``every $\mathbf{\Pi}_2^0$ subset of $\mc P(\N)$ is separable'' implies $\PCAo$.
\end{cor}
\begin{proof}
The argument is identical to the proof of Proposition \ref{equi-overt-PCAo}.
\end{proof}

\begin{prop}
Within $\ACAo$, the following assertions hold for the universal quasi-metric space $\mc P(\N)$.
\begin{enumerate}
\item Every closed subset of $\mc P(\N)$ is separable.
\item Every $G_{\delta}$ subset of $\mc P(\N)$ is separable.
\item Every $\mathbf{\Sigma}_2^0$ subset of $\mc P(\N)$ is separable.
\end{enumerate}
\end{prop}

\begin{proof}
We work within $\ACAo$ and provide the construction for each case:
\begin{enumerate}
\item Let $(0,1,\langle c_i:i\in\N\rangle)$ where $c_i\in\mc P_{\fin}(\N)$ for all $i\in\N$ be a code for a closed set $K\subseteq\mc P(\N)$. We define the set of ``surviving'' finite sequences as $D:=\{\sigma\in\mc P_{\fin}(\N):\forall i(c_i\not\subseteq\sigma)\}$, which is definable by arithmetical comprehension. Since every point in $K$ must avoid all forbidden basic open sets coded by $c_i$, the set of all infinite extensions of sequences in $D$ that do not hit $\mc P(\N)\setminus K$ forms a dense subset of $K$.
\item Let $(0,2,\langle (0,\emptyset),(1,g_{i,m}):i,m\in\N\rangle)$ where $g_{i,m}\in\mc P_{\fin}(\N)$ for all $i,m\in\N$ be a code for a $G_{\delta}$ set, representing the intersection of open sets $U_i=\bigcup_{m\in\N}N_{g_{i,m}}$. We define the set of consistent finite sequences as
 $D:=\{\sigma\in2^{<\N}:\forall i\, \exists m \,\forall n\in g_{i,m}(n\in\sigma\lor \lh(\sigma)\leq n)\}$ by arithmetical comprehension. For each $\sigma \in D$, we can effectively construct a point $x_\sigma = \sigma \text{\textasciicircum} \langle 1, 1, \dots \rangle$.  This $x_\sigma$ satisfies the $G_\delta$ condition. The collection of such points $\{x_\sigma : \sigma \in D\}$ constitutes the required dense sequence.
\item Let $(1,2,\langle(1,u_{i,m}),(0,v_{i,m}):i,m\in\N\rangle)$ be a code for a $\mathbf{\Sigma}_2^0$ set, represented as a union of sets $B_i\cap C_i$ where $B_i$ are open and $C_i$ are closed. Specifically, $B_i=\bigcup_{m\in\N} N_{u_{i,m}}$, $C_i=X\setminus \bigcup_{m\in\N} N_{v_{i,m}}$. By arithmetical comprehension, we can define the set of points $D:=\{\sigma\in\mc P_{\fin}(\N):\exists i,m\,\forall n\,((\sigma\in N_{u_{i,m}})\land(\sigma\notin N_{v_{i,n}}))\}$. Since each $B_i \cap C_i$ is an intersection of an open set and a closed set, and $\ACAo$ allows us to pick witnesses for the existence of $i$, the union of the dense sequences for each $X_i$ (whenever $X_i \neq \emptyset$) yields a dense sequence for the entire $\mathbf{\Sigma}_2^0$ set. \qedhere
\end{enumerate}
\end{proof}

Finally, we establish that the assertion ``every UF space is homeomorphic to a quasi-Polish space'' also derives $\PCAo$.

\begin{thm}
The following statements are equivalent over $\RCAo$.
\begin{enumerate}
\item $\PCAo$.
\item Every UF space is homeomorphic to a second-countable Smyth-complete quasi-metric space.
\end{enumerate}
\end{thm}
\begin{proof}
This proof follows the same strategy as \cite[Thm 2.12]{brown1990notions}.
$(1)\Rightarrow(2)$: By Corollary \ref{UFtoPi20}, every UF space is homeomorphic to a $\mathbf{\Pi}_2^0$ subset of $\mc P(\N)$. Since $\mc P(\N)$ is a second-countable Smyth-complete quasi-metric space, and it is provable in $\PCAo$ that every separable $\mathbf{\Pi}_2^0$ subset of a second-countable Smyth-complete quasi-metric space is a second-countable Smyth-complete quasi-metric space by Propositions \ref{analyseparable} and \ref{sepPi20qP}, the result follows.

$(2)\Rightarrow(1)$: We show that $(2)$ implies $\PCAo$ over $\RCAo$.  
First, we establish $\ACAo$ by showing that the range of any injective function $f:\N\to\N$ exists. 

 We encode the sequence in $2^{<\mathbb{N}}$ as $\langle\sigma_n\rangle_n$, where each string $\sigma_n$ of length $n+1$ is defined by $\sigma_n(n)=1$ and $\sigma_n(m)=0$ for all $m < n$. We define a tree $T$ as follows:
\[
T:=\{\sigma\in 2^{<\N}:\forall m<\lh(\sigma)\,\exists k<\lh(\sigma_{f(m)})(\sigma(k)\neq\sigma_{f(m)}(k))\}.
\]
We then define a poset $P=T\cup\{s_n:n\in\N\}$
with the order generated by:
\begin{align*}
\sigma\geq_P\tau\quad &\text{if } (\sigma\preceq\tau)\land \sigma,\tau\in T;\\
\sigma\geq_Ps_n \quad&\text{if } \exists \tau\in 2^{<n}(\sigma\preceq\tau\land\forall k<2(\tau\string^\langle k\rangle\notin T)\land \sigma,\tau\in T);\\
s_n\geq_Ps_m\quad&\text{if } n\leq m.
\end{align*}
By assumption $(2)$, $\UF(P)$ is homeomorphic to a second-countable Smyth-complete quasi-metric space $X$ with code $(A,d)$; let $\phi:\UF(P)\cong X$ denote this homeomorphism.
 Let $A=\langle a_i:i\in\N\rangle$. Without loss of generality, we may assume that $a_0=\phi(\langle s_n\rangle_{n\in\N})$. Within $\RCAo$, we define a subset $A^{\prime}\subseteq A$ as:
\[
A^{\prime}=\{x\in A:\hat{d}(a_0,x)>0\}.
\] This $A^{\prime}$ is a $\hat{d}$-dense set of $X\backslash N_{s_0}$ which is homeomorphic to the body of the tree $[T]$. We denote $A^{\prime}=\langle x_k:k\in\N\rangle$. Let $B=\{n\in\N:\sigma_n\in [T]\}=\{n\in\N:\exists k(\sigma_n\preceq x_k)\}$. By $\Delta_1^0$-comprehension, $B$ exists within $\RCAo$. We have $\forall n(n\notin B\leftrightarrow\exists m(f(m)=n))$ so $\ACAo$ holds.
 
Next, we establish $\PCAo$ over $\ACAo$. Let $\langle T_i\rangle_{i\in\N}$ be a given sequence of subtrees of $\N^{<\N}$. For each $i$, we define a poset $P_i:=T_i\cup\{a_n:n\in\N\}$
with the order generated by:
\begin{align*}
\sigma\geq_{P_i}\tau \quad &\text{if }  \sigma\preceq\tau;\\
\sigma\geq_{P_i}a_n\quad &\text{if }  \exists\tau\in\N^{<n}(\sigma\preceq\tau\land\forall k(\tau\string^\langle k\rangle\notin T_i)\land \sigma,\tau\in T_i);\\
a_n\geq_{P_i}a_m\quad&\text{if } n\leq m.
\end{align*}
Let $P:=\bigsqcup_{i\in\N}P_i$ be the disjoint union of these posets. By assumption $(2)$, $\UF(P)$ is homeomorphic to a second-countable Smyth-complete quasi-metric space $X_A$ represented by the code $(A,d)$. Let $\phi:X_A\to\UF(P)$ be the homeomorphism. We define the set $K=\{i\in\N:\exists j(x_j\in A\land \phi(x_j)\in\UF(P_i)\setminus \{F_{\infty}\}\}$ where $F_{\infty}=\upcl(\langle a_n\rangle_n)$ within $\ACAo$. We can verify that if $T_i$ has no infinite path, the only unbounded filter in $\UF(P_i)$ is $F_{\infty}$. Conversely, any infinite path $x$ of $T_i$ corresponds to an unbounded filter in $\UF(P_i)$ that is distinct from $F_{\infty}$.
Since $X_A$ is a second-countable Smyth-complete quasi-metric space, the existence of such points $x$ (and thus the existence of paths in $T_i$) can be determined using the dense set $A$ and the completeness of $d$. Specifically, $K=\{i\in\N:T_i\text{ has an infinite path}\}$. Since the existence of $K$ is guaranteed by the properties of $X$ in $\RCAo$, $\PCAo$ is proved.
\end{proof}

As a consequence of the equivalence between UF spaces and quasi-Polish spaces, the statement that the class of countably based UF spaces is closed under countable products is readily obtained. However, a direct proof is available within $\ACAo$. 
\begin{prop}[Formalized version of Proposition \ref{1-product-UF}] The following statement is provable within $\ACAo$.
For any pair of countable posets $P$ and $Q$, there exists a countable poset $R$ such that $\UF(P)\times \UF(Q)$ is homeomorphic to $\UF(R)$. Similarly, for any sequence of countable posets $\langle P_n:n\in\N\rangle$, there exists a countable poset $R$ such that $\prod_{n\in\N}\UF(P_n)\cong\UF(R)$.
\end{prop}

\begin{proof}
We can safely assume within $\RCAo$ that $P\subseteq\N$ and $Q\subseteq\N$ are countable handy posets. Let $R=P\times Q$ be the poset equipped with the product order: $(p,q)\leq_R(p^{\prime},q^{\prime})\Leftrightarrow (p<_Pp^{\prime}\land q<_Qq^{\prime})\lor(p=p'\land q=q')$. We define the continuous mappings between these spaces by their formal codes. Let $\Phi$ be the code for $f:\UF(P)\times\UF(Q)\to\UF(R)$, and let $\Psi_1$, $\Psi_2$ be the codes for the projections $g_1:\UF(R)\to\UF(P)$ and $g_2:\UF(R)\to\UF(Q)$, respectively:
\begin{align*}
\Phi&=\{(m,(p,q),(p,q)):m\in\N,p\in P\land q\in Q\};\\
\Psi_1&=\{(m,(p,q),p):m\in\N,p\in P\land q\in Q\};\\
\Psi_2&=\{(m,(p,q),q):m\in\N,p\in P\land q\in Q\}.
\end{align*}
For any unbounded filters $F\in\UF(P)$ and $G\in\UF(Q)$, the set $f(F,G):=\{(p,q)\in R:p\in F\land q\in G\}$ is an unbounded filter of $R$. The continuity of $f$ is verified by observing that for any open set $O_U\subseteq\UF(R)$ coded by $U\subseteq\N\times P\times Q$, the code of the preimage $f^{-1}(O_U)$ is $\{((m,p),(m,q))\in\N\times P\times\N\times Q:\exists(m,(p,q))\in U\}$ can be constructed via $\Sigma_0^0$-comprehension. Thus, $f$ is a total continuous function.
The injectivity of $f$ follows directly from its definition. To show surjectivity, let $H\in\UF(R)$ be an unbounded filter. We have $g_{1}(H)=\{p\in P:\exists q\in Q(p,q)\in H\}$ and $g_{2}(H)=\{q\in Q:\exists p\in P(p,q)\in H\}$. These $g_{1}(H)$ and $g_{2}(H)$ are unbounded filters on $P$ and $Q$, respectively. 
Since $H$ is a filter on the product $P\times Q$, it is easily verified that $f(g_1(H),g_2(H))=H$. The continuity of $g_1$ and $g_2$ is similarly established by checking the preimages of open sets in $\UF(P)$ and $\UF(Q)$.
Therefore, $f$ is a homeomorphism between $\UF(P)\times\UF(Q)$ and $\UF(R)$, completing the proof.
The case for a sequence of posets follows by a similar construction:
\[
R:=\{\langle p_i:i<k\rangle:k\in\N, \forall i<k(p_i\in P_i)\}
\]
 with the ordering $\leq_R$ such that
\begin{align*}
\langle p_i:i<k\rangle\leq_R\langle q_i:i<h\rangle\iff& (k>h\land \forall i<k\,(p_i<_{P_i}q_i))\\
&\lor(k=h\land \forall i<k(p_i=q_i)).
\end{align*}
\end{proof} 

\subsection{Countably presented frames}

Since a general treatment of all frames is beyond the scope of second-order arithmetic, we restrict our study to countably presented frames.

\begin{definition}[Countably presented frames]
The following definitions are formulated within $\mathsf{RCA}_0$.
\begin{itemize}
    \item A \textit{countably presented frame} $\langle G; R \rangle$ is a pair consisting of a set of generators $G = \{ g_i \}_{i \in \mathbb{N}}$ and a set of relations $R = \{ (u_j, v_j) \}_{j \in \mathbb{N}}$, where each $u_j$ and $v_j$ is an expression over $G$.
    
    \item An \textit{expression} $a$ of the frame $\langle G; R \rangle$ is a sequence $\langle f_k \rangle_{k \in \mathbb{N}}$ of finite subsets of $G$ (i.e., $f_k \in \mathcal{P}_{\fin}(G)$). Such an expression $a$ is denoted by:
    \[
   a= \bigvee_{k \in \mathbb{N}} \left( \bigwedge_{g \in f_k} g \right).
    \] We write $a\in \langle G\rangle$ or $a\in\langle G; R \rangle$ to mean that $a$ is an expression of $\langle G; R \rangle$. These $\langle G\rangle$ and $\langle G;R\rangle$ are not formal objects, since they denotes classes.
    \item We define formal operations $\bigvee$ and $\bigwedge$ on $\langle G\rangle$ so as to satisfy the infinite distributive law:
    \[
    a \land \left( \bigvee_{k \in \mathbb{N}} b_k \right) = \bigvee_{k \in \mathbb{N}} (a \land b_k).
    \]
    Then, $\langle G\rangle$ can be regarded as a lattice with arbitrary joins.    
    \item For two expressions $a = \bigvee_{i \in \mathbb{N}} \bigwedge_{g \in p_i} g$ and $b = \bigvee_{j \in \mathbb{N}} \bigwedge_{g \in q_j} g$, the preorder $a \leq b$ is defined by:
    \[
    \forall i \in \mathbb{N}\, \exists j \in \mathbb{N}\, (p_i \supseteq q_j).
    \]
    (Note: $p_i \supseteq q_j$ as sets of generators indicates that the meet over $p_i$ is smaller than the meet over $q_j$ in the lattice order.)
    
\end{itemize}
\end{definition}

To show that the point sets of countably presented frames are homeomorphic to $\mathbf{\Pi}_2^0$ subspaces of $\mc P(\N)$, we characterize the minimal order relation induced by the frame structure using two distinct formulations, $\preceq$ and $\vdash$. $\preceq$ is the most direct way to represent the least congruence preorder. 
The relation $\vdash$ comes from geometric proof theory \cite{deBrecht-unpublished}. See also Fourman/Grayson \cite{fourman1982formal}.


For a set of countable relations $R$, we write $u\in\mathsf{field}(R)$ if $\exists v\in Q((u,v)\in R\lor(v,u)\in R)$. The class $\mathsf{field}(R)$ is enumerable within $\RCAo$.

\begin{defi}[congruence preorder]
The following definitions are formulated within $\mathsf{RCA}_0$.
A set $(Q,\preceq_Q)$ generates a \textit{congruence preorder} $\preceq_m$ of $\langle G;R\rangle$ if it satisfies the following conditions:
\begin{enumerate}
\item  $Q$ is a countable subset of expressions $\{ w_i:i\in\N \}\subseteq\langle G\rangle$ such that $\mc P_{\fin}(G)\subseteq Q$ and $\mathsf{field}(R)\subseteq Q$. 
\item ${\preceq_Q}\subseteq Q\times Q$ is a binary relation on $Q$.
\item $R\subseteq {\preceq_Q}$.
\item $q\preceq_Q q$ for all $q\in Q$. 
\end{enumerate}
Also, for $\langle p_i\rangle_{i\in\N}\subseteq\mc P_{\fin}(G)$, we define $\varphi(Q,\langle p_i\rangle_{i\in\N})$ to be the following condition:
\begin{enumerate}
\item $\forall i\,(p_i\subseteq p_{i+1})$;
\item $\forall c,d\in Q\,\forall i\,(p_i\leq c\preceq_Qd\rightarrow\exists k(p_k\leq d))$.
\end{enumerate}
For $a,b\in\langle G;R\rangle $, $a\preceq_m b$ means: 
\begin{gather*}
\forall \langle p_i\rangle_{i\in\N}\subseteq\mc P_{\fin}(G)\,((\varphi(Q,\langle p_i\rangle_{i\in\N})\land p_0\leq a)\rightarrow\exists i\exists r,l\in Q(p_i\leq r\preceq_Ql\leq b)).
\end{gather*}
\end{defi}

The congruence preorder $\preceq_m$ generated by $Q=\mc P_{\fin}(G)\cup \mathsf{field}(R)$ and ${\preceq_Q}=\{(q,q):q\in Q\}\cup R$ is the least congruence preorder; we denote it by $\preceq$. 


\begin{lem}
Within $\RCAo$, every congruence preorder $\preceq_m$ in the above definition satisfies the conditions of congruence preorder in ordinary mathematics: 
\begin{enumerate}
\item $a\leq b\Rightarrow a\preceq_m b$.
\item $a\preceq_m b, b\preceq_m c\Rightarrow a\preceq_m c$.
\end{enumerate}
where $a,b,c\in\langle G\rangle$.
\end{lem}
\begin{proof}
\begin{enumerate}
\item Let us assume $a\leq b$ for $a,b\in\langle G\rangle$. Then if a sequence $\langle p_i\rangle_{i\in\N}$ satisfies $p_0\leq a\land\varphi(Q,\langle p_i\rangle_{i\in\N})$, then $p_0\leq a\leq b$. In particular, $p_0\leq p_0\preceq_Qp_0\leq b$, so $a\preceq_m b$. 
\item Let us assume $a\preceq_m b, b\preceq_m c$. Every sequence $\langle p_i\rangle_{i\in\N}$ such that $p_0\leq a\land\varphi(Q,\langle p_i\rangle_{i\in\N})$ has $i\in\N$ and $r,l\in Q$ such that $p_i\leq r\preceq_Q l\leq b$. Then there is a $k\in\N$ such that $p_k\leq l$ by the assumption of $\langle p_i\rangle_{i\in\N}$. Then take $\langle q_m\rangle_{m\in\N}$ by $q_{m}=p_{m+k}$. This $\langle q_m\rangle_m$ satisfies $q_0\leq b\land\varphi(Q,\langle q_m\rangle_{m\in\N})$. So $\exists h,f\in Q\,\exists j(q_j\leq h\preceq_Q f\leq c)$. So we have $p_{k+j}\leq h\preceq_Q f\leq c$. \qedhere
\end{enumerate}
\end{proof}
Also, this definition yields the expected closure properties: if $a_i\preceq_m b$ for all $i\in I$, then $\bigvee_{i\in I}a_i\preceq_m b$; and if $a\preceq_m b_i$ for all $i\in I$ with $I$ finite, then $a\preceq_m \bigwedge_{i\in I}b_i$.



\begin{defi}[Proof-theoretic formulation of $\vdash$, $\RCAo$]
Let $a=\bigvee p_k$, $b=\bigvee q_j$ be expressions in a countably presented frame $\langle G;R\rangle$.
We define the \textit{deductive relation} $a\vdash b$ to hold if and only if there exists a well-founded tree $T\subseteq\N^\N$ and a sequence of pairs of expressions $\langle (a_\sigma,b_\sigma):\sigma\in T\rangle$ such that for every node $\sigma\in T$, the pair $(a_\sigma,b_\sigma)$ satisfies one of the following conditions:
\begin{enumerate}
\item There is $(u,v)\in R$ and an expression $c$ such that $(a_\sigma,b_\sigma)=(c\land u,c\land v)$, or $a_\sigma\leq b_\sigma$ (Base case).
\item There is an expression $c$ such that $(a_{\sigma\string^\langle0\rangle},b_{\sigma\string^\langle0\rangle})=(a_\sigma,c)$ and $(a_{\sigma\string^\langle1\rangle},b_{\sigma\string^\langle1\rangle})=(c,b_\sigma)$ (Cut).
\item If $a_\sigma=\bigvee_{i\in\N}p_i$, then for each $i$ the child $\sigma\string^\langle i\rangle$ belongs to $T$ and $(a_{\sigma\string^\langle i\rangle},b_{\sigma\string^\langle i\rangle})=(p_i,b_\sigma)$ (Left Join-Introduction).
\item If $b_\sigma=\bigvee_{h\in\N}q_h$, then there exist an $h$ and a child $\sigma\string^\langle h\rangle\in T$ such that $(a_{\sigma\string^\langle h\rangle},b_{\sigma\string^\langle h\rangle})=(a_\sigma,q_h)$ (Right Join-Introduction).
\end{enumerate}
The deduction is successful if the root $\langle\rangle \in T$ corresponds to the initial pair $(a,b)$; i.e., $a_{\langle\rangle}=a$ and $b_{\langle\rangle}=b$.

\end{defi}

We have now defined two types of order relations for the frame $\langle G; R \rangle$. For the rest of this section, we shall establish that these are spatial.

\begin{defi}[Spatiality]
The following definition is formulated within $\mathsf{RCA}_0$.
A countably presented frame $\langle G; R \rangle$ with the preorder $\triangleleft$ is said to be \textit{spatial} if there exists a topological space $X$ and an isomorphism $f:\langle G;R\rangle\to \mc O(X)$ (where $\mc O(X)$ denotes the frame of open sets on $X$, each of whose members is a code for an open set) such that for every pair of expressions $(a,b)$, the following equivalence holds: 
\[
a \triangleleft b\Leftrightarrow \forall x\in X(x\in O_{f(a)}\rightarrow x\in O_{f(b)}).
\]
 where $O_{f(a)}$ and $O_{f(b)}$ denotes the open set coded by $f(a)$ and $f(b)$, respectively.
Equivalently, the frame is spatial if any failure of the deductive relation is witnessed by a point:
\[
a \ntriangleleft b \Leftrightarrow \exists x\in X(x\in O_{f(a)}\land x\notin O_{f(b)}).
\]
\end{defi}

\begin{prop}\label{precspatial}
Over $\RCAo$, every countably presented frame $\langle G;R\rangle$ with the preorder $\preceq$ is spatial. Specifically it is isomorphic to the frame of open sets of some $\mathbf{\Pi}_2^0$ subspace of $\mc P(\N)$.
\end{prop}
\begin{proof}
Let $f:G\to\N$ be a bijection mapping basic generators $g_i$ to natural numbers $i$. This induces an isomorphism of power sets $f:\mc P(G)\to\mc P(\N)$. Define the subspace $X\subseteq\mc P(\N)$ by:
\[
X=\bigcap_{i\in\N,(u_i, v_i)\in R}(O_{f(v_i)}\cup(\mc P(\N)\setminus O_{f(u_i)})).
\] We show that $a\npreceq b\Leftrightarrow \exists x\in X(x\in O_{f(a)}\land x\notin O_{f(b)})$ for all expressions $a,b$.
\begin{itemize}
\item $a\npreceq b\Rightarrow \exists x\in X(x\in O_{f(a)}\land x\notin O_{f(b)})$: Assume $a\npreceq b$. Unwinding the definition of $a\npreceq b$, there exists a congruence preorder generated by $(Q,\preceq_Q)$ and a sequence $\langle p_i\rangle_{i\in\N}\subseteq\mc P_{\fin}(G)$ such that  $(p_0\leq a)\land\varphi(Q,\langle p_i\rangle_{i\in\N})$ and $\forall i\,\forall r,l\in Q\lnot(p_i\leq r\preceq_Ql\leq b)$. Indeed, the decreasing sequence $\langle p_i\rangle_{i\in\N}$ represents a point $y \in X$ and $y\in O_{f(a)}\setminus O_{f(b)}$.
\item $\exists x\in X(x\in O_{f(a)}\land x\notin O_{f(b)})\Rightarrow a\npreceq b$: 
Assume there exists $x \in X$ such that $x \in f(a) \land x \notin f(b)$. Let $x$ be represented by the sequence $\langle q_i\rangle_{i \in \mathbb{N}}$. Let $(Q,\preceq_Q)$ be $Q=\mc P_{\fin}(G)\cup \mathsf{field}(R)$ and ${\preceq_Q}=\{(q,q):q\in Q\}\cup R$. Since $x\in f(a)$, there exists $k$ such that $q_{k}\leq a$. Let us define $\langle p_i\rangle_{i\in\N}$ by $p_i=q_{i+k}$. This $\langle p_i\rangle_{i\in\N}$ satisfies the properties: $(p_0\leq a)$, $\varphi(Q,\langle p_i\rangle_{i\in\N})$ and $\forall i\,\forall r,l\in Q\,\lnot(p_i\leq r\preceq_Ql\leq b)$. \qedhere
\end{itemize}
\end{proof}

\begin{rem}[Points of a frame]

In Section 2, we discussed the topological space $X$ whose set of points consists of all frame homomorphisms $\xi: \langle G; R \rangle \to \{0,1\}$. Within second-order arithmetic, we can also define a congruence preorder induced by a point $x \in X$.

A \textit{point of a frame generated by $(G,R)$} is a decreasing sequence $x = \langle p_i\rangle_{i\in\N}$ of elements of $\mathcal{P}_{\mathrm{fin}}(G)$ that satisfies:
\[
\forall i \, \exists j \, (p_i \leq u \to p_j \leq v)
\]
for all $(u,v) \in R$.
For a point $x = \langle p_i\rangle_{i\in\N}$, we define a preordered set $(Q, \preceq_Q)$ as follows:
\begin{align*}
Q &:= \mathcal{P}_{\mathrm{fin}}(G) \cup \mathsf{field}(R), \\
p \preceq_Q q &\iff  \exists k \, (p_k \leq p) \to \exists h \, (p_h \leq q) .
\end{align*}
The relation $(Q, \preceq_Q)$ generates a congruence preorder $\preceq_x$.

When $Y= \{ g \in G : \exists i \, (p_i \leq g) \}$ exists and $\xi: G \to \{0,1\}$ is the characteristic function of $Y$, the induced relation $\preceq_x$ can be expressed as
\[
a \preceq_x b \iff \forall i \, (\xi(r_i) = 0) \lor \exists j \, (\xi(l_j) = 1)
\]
for any two expressions $a = \langle r_i \in \mathcal{P}_{\mathrm{fin}}(G) : i \in \mathbb{N} \rangle$ and $b = \langle l_j \in \mathcal{P}_{\mathrm{fin}}(G) : j \in \mathbb{N} \rangle$.
In other words, the characteristic function $\xi: G \to \{0,1\}$ of $X$ corresponds to a frame homomorphism from $\langle G; R \rangle$ to the two-element frame $\mathbf{2} = \mathcal{P}(1)$. 

Furthermore, let us define $\subseteq_{\mathrm{pt}}$ by setting $a \subseteq_{\mathrm{pt}} b$ if and only if $a \preceq_x b$ for every point $x = \langle p_i\rangle_{i\in\N}$ of $\langle G; R \rangle$. Since a point $x = \langle p_i\rangle_{i\in\N}$ of a frame generated by $(G,R)$ satisfies $\varphi(Q, \langle p_i\rangle_{i\in\N})$, we can verify $a \subseteq_{\mathrm{pt}} b \iff a \preceq b$ within $\mathsf{RCA}_0$.
\end{rem}

\begin{cor}
Within $\RCAo$, a second-countable $T_0$ space $X$ is homeomorphic to a $\mathbf{\Pi}_2^0$ subspace of $\mc P(\N)$ if and only if its frame of open sets is isomorphic to a countably presented frame $\langle G;R\rangle$ with $\preceq$.
%
\end{cor}
\begin{proof}
Every countably presented frame with $\preceq$ can be written as an open family of a $\mathbf{\Pi}_2^0$ subspace of $\mc P(\N)$ by Proposition~\ref{precspatial}. 
 Conversely, let $X=\bigcap_{i\in\N}(O_{V_i}\cup(\mc P(\N)\setminus O_{U_i}))$ be a $\mathbf{\Pi}_2^0$ subspace of $\mc P(\N)$. The open family of $X$ is clearly presented by a frame of the form $\langle \N;R\rangle$, where $R$ consists of the relations $(u_i,v_i)$ where each $u_i,v_i$ is the expression associated with a code of an open set $U_i,V_i$, respectively. The formalization of this transition follows the same logic as in Proposition \ref{precspatial}, ensuring that the construction remains effective within $\RCAo$.
\end{proof}

\begin{prop}\label{proofspatial}
Over $\ACAo$, every countably presented frame $\langle G;R\rangle$ with the preorder $\vdash$ is spatial. Specifically it is isomorphic to the frame of open sets of some $\mathbf{\Pi}_2^0$ subspace of $\mc P(\N)$.
\end{prop}
\begin{proof}
Let $f:G\to\N$ be a bijection mapping basic generators $g_i$ to natural numbers $i$. This induces an isomorphism of power sets $f:\mc P(G)\to\mc P(\N)$. Define the subspace $X\subseteq\mc P(\N)$ by:
\[
X=\bigcap\{O_{f(v_i)}\cup(\mc P(\N)\setminus O_{f(u_i)}):i\in\N,(u_i, v_i)\in R\}.
\] We show that $a\not\vdash b\Leftrightarrow \exists x\in X(x\in O_{f(a)}\land x\notin O_{f(b)})$ for all expressions $a,b$.

\begin{itemize}
\item $a\not\vdash b\Rightarrow \exists x\in X(x\in O_{f(a)}\land x\notin O_{f(b)})$: Suppose $a\not\vdash b$.
We construct a proof tree $T$. Starting with $(a_{\langle\rangle},b_{\langle\rangle})=(a,b)$, we define the successors based on the deduction rules:
\begin{itemize}
\item If $a_{\sigma}=\bigvee_{i\in\N}p_i$, we branch to $(p_i,b_\sigma)$ for each $i$.
\item If $a_{\sigma}\in\mc P_{\fin}(G)$ and there exists a relation $(u_i,v_i)\in R$ such that $(a_{\sigma}\leq u_i)\land (a_{\sigma}\not\leq v_i)\land i<\lh(\sigma)$, we branch to $(a_\sigma, a_\sigma\land v_i)$ and $(a_\sigma\land v_i,b_\sigma)$ for the least such $i$.
\item Otherwise, we branch to $(a_\sigma,b_\sigma)$ and $(a_\sigma\land g_m,b_\sigma)$ for all $m\in\N$ to ensure the path determines a point.
\end{itemize}
Since $a\not\vdash b$, the tree $T$ is not well-founded. There exists an infinite path
$y$ through $T$. This path induces a point $x=\langle r_i\rangle_{i\in\N}$, where $r_i=\{n\in\N:n<i\land a_{y[i]}\leq g_n\}$. 
By construction, $x$ satisfies all relations in $R$, implying $x\in X$. Furthermore, $x\in O_{f(a)}\land x\notin O_{f(b)}$, providing the required witness.
\item $\exists x\in X(x\in O_{f(a)}\land x\notin O_{f(b)})\Rightarrow a\not\vdash b$: Assume there exists $x=\langle r_i\rangle_{i\in\N}$ in $X$ such that $x\in O_{f(a)}\land x\notin O_{f(b)}$. Using $\ACAo$, we define a characteristic function $\xi:G\to\{0,1\}$ representing the point $x$; i.e., for all $n\in\N$, $\xi(g_n)=1$ if $\exists i(f(r_i)\leq g_n)$ and $\xi(g_n)=0$ otherwise. For any proof tree $S$ for $a \vdash b$, we can use $\xi$ to construct an infinite path $y$ through $S$. At each step, suppose $y\rest n=\sigma$ and the corresponding pair is $(a_\sigma,b_\sigma)$. We choose a successor $\sigma\string^\langle m\rangle$ such that the corresponding pair $(a_{\sigma\string^\langle m\rangle},b_{\sigma\string^\langle m\rangle})$ still satisfies $\exists i\,\xi(p_i)=1$ and $\forall j\,\xi(q_j)=0$, where $a_{\sigma\string^\langle m\rangle}=\bigvee_{i\in\N}p_i$ and $b_{\sigma\string^\langle m\rangle}=\bigvee_{j\in\N}q_j$. Such a choice is always possible by induction on $n$ and the properties of points in $X$. Therefore, no such well-founded proof tree exists, and $a \not\vdash b$. \qedhere
\end{itemize} 
\end{proof}

The logical strength of the above proposition has not yet been determined.

\begin{question}
Does Proposition \ref{proofspatial} actually require $\ACAo$ over $\RCAo$?
\end{question}

\section*{Summary}
We have established the equivalence of representations for quasi-Polish spaces and have characterized their logical strength within the hierarchy of second-order arithmetic. The following diagram summarizes the implications and the axiomatic requirements for the correspondences between quasi-Polish spaces and their representations.


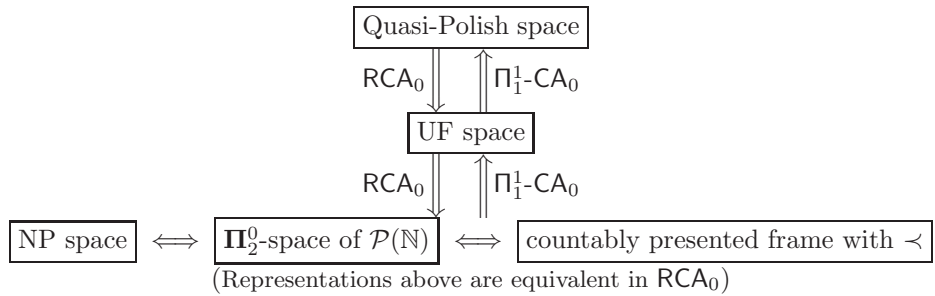
\begin{equation*}
\begin{array}{c}
    \fbox{Quasi-Polish space} \\[1ex]
    \mathsf{RCA}_0 \Big\Downarrow \quad \Big\Uparrow \mathsf{\Pi}_1^1\text{-}\mathsf{CA}_0 \\[1.5ex]
    \fbox{UF space} \\[1ex]
    \mathsf{RCA}_0 \Big\Downarrow \quad \Big\Uparrow \mathsf{\Pi}_1^1\text{-}\mathsf{CA}_0 \\[1.5ex]
    \fbox{NP space} \Leftrightarrow \fbox{Ideal space} \Leftrightarrow \fbox{$\mathbf{\Pi}_2^0\text{-subspace of } \mathcal{P}(\mathbb{N})$} \Leftrightarrow \fbox{Countably presented frame with $\preceq$} \\[1.5ex]
    {\small (\text{The representations above are equivalent over } \mathsf{RCA}_0)}
\end{array}
\end{equation*}

\section*{Acknowledgements}
We would like to thank Matthew de Brecht for reading a preliminary version of this paper and providing valuable comments.
We would like to thank Takayuki Kihara for useful discussions.

This work was partially supported by JSPS KAKENHI grant numbers JP21KK0045, JP23K03193 and 26K00615, as well as by the Research Institute for Mathematical Sciences,
an International Joint Usage/Research Center located in Kyoto University.

\bibliographystyle{plain}
\bibliography{bib-all}

\begin{thebibliography}{10}

\bibitem{benham2024ginsburgSands}
Heidi Benham, Andrew De~Lapo, Damir Dzhafarov, Reed Solomon, and Java~Darleen
  Villano.
\newblock The {G}insburg--{S}ands theorem and computability theory, 2024.

\bibitem{brown1990notions}
Douglas~K Brown.
\newblock Notions of closed subsets of a complete separable metric space in
  weak subsystems of second-order arithmetic.
\newblock {\em Logic and computation (Pittsburgh, PA, 1987)}, 106:39--50, 1990.

\bibitem{deBr2013}
Matthew de~Brecht.
\newblock Quasi-{P}olish spaces.
\newblock {\em Ann. Pure Appl. Logic}, 164(3):356--381, 2013.

\bibitem{de_Brecht_2018}
Matthew de~Brecht.
\newblock A generalization of a theorem of hurewicz for quasi-polish spaces.
\newblock {\em Logical Methods in Computer Science}, Volume 14, Issue 1,
  February 2018.

\bibitem{de2019note}
Matthew de~Brecht.
\newblock A note on the spatiality of localic products of countably based sober
  spaces.
\newblock {\em CCC 2019: Computability, Continuity, Constructivity-from Logic
  to Algorithms}, page~15, 2019.

\bibitem{de2020some}
Matthew de~Brecht.
\newblock Some notes on spaces of ideals and computable topology.
\newblock In {\em Conference on Computability in Europe}, pages 26--37.
  Springer, 2020.

\bibitem{deBrecht-unpublished}
Matthew de~Brecht.
\newblock Notes on formal geometric proofs.
\newblock 2025.

\bibitem{BKS2024}
Matthew De~Brecht, Takayuki Kihara, and Victor Selivanov.
\newblock Ideal presentations and numberings of some classes of effective
  quasi-{Polish} spaces.
\newblock {\em Computability}, 13(3-4):325--348, 2024.

\bibitem{de2020overt}
Matthew de~Brecht, Arno Pauly, and Matthias Schr{\"o}der.
\newblock Overt choice.
\newblock {\em Computability}, 9(3-4):169--191, 2020.

\bibitem{dorais2011compactCountable}
Fran{\c c}ois~G. Dorais.
\newblock Reverse mathematics of compact countable second-countable spaces,
  2011.

\bibitem{fernandez2023caristi}
D.~Fern\'andez-Duque, P.~Shafer, H.~Towsner, and K.~Yokoyama.
\newblock Metric fixed point theory and partial impredicativity.
\newblock {\em Philos. Trans. Roy. Soc. A}, 381(2248):Paper No. 20220012, 20,
  2023.

\bibitem{fernandez2020ekeland}
David Fern\'andez-Duque, Paul Shafer, and Keita Yokoyama.
\newblock Ekeland's variational principle in weak and strong systems of
  arithmetic.
\newblock {\em Selecta Math. (N.S.)}, 26(5):Paper No. 68, 38, 2020.

\bibitem{fourman1982formal}
M.~P. Fourman and R.~J. Grayson.
\newblock Formal spaces.
\newblock In {\em The {L}. {E}. {J}. {B}rouwer {C}entenary {S}ymposium
  ({N}oordwijkerhout, 1981)}, volume 110 of {\em Stud. Logic Found. Math.},
  pages 107--122. North-Holland, Amsterdam, 1982.

\bibitem{frittaion2016noetherian}
Emanuele Frittaion, Matthew Hendtlass, Alberto Marcone, Paul Shafer, and Jeroen
  Van~der Meeren.
\newblock Reverse mathematics, well-quasi-orders, and {N}oetherian spaces.
\newblock {\em Arch. Math. Logic}, 55(3-4):431--459, 2016.

\bibitem{genovesi2026regularCSCS}
Giorgio~G. Genovesi.
\newblock Reverse mathematics of regular countable second countable spaces.
\newblock {\em J. Symbolic Logic}, 2026.
\newblock Online first.

\bibitem{heckmann2015spatiality}
Reinhold Heckmann.
\newblock Spatiality of countably presentable locales (proved with the baire
  category theorem).
\newblock {\em Mathematical Structures in Computer Science}, 25(7):1607--1625,
  2015.

\bibitem{honda-unpublished}
Tadayuki Honda.
\newblock Unpublished note.
\newblock 2021.

\bibitem{kechris2012classical}
A.~Kechris.
\newblock {\em Classical Descriptive Set Theory}.
\newblock Graduate Texts in Mathematics. Springer New York, 2012.

\bibitem{kunzi1983strongly}
HPA K{\"u}nzi.
\newblock On strongly quasi-metrizable spaces.
\newblock {\em Archiv der Mathematik}, 41(1):57--63, 1983.

\bibitem{mummert2005reverse}
Carl Mummert.
\newblock {\em On the reverse mathematics of general topology}.
\newblock The Pennsylvania State University, 2005.

\bibitem{mummert2006reverseMF}
Carl Mummert.
\newblock Reverse mathematics of {MF} spaces.
\newblock {\em J. Math. Log.}, 6(2):203--232, 2006.

\bibitem{mummert2010topological}
Carl Mummert and Frank Stephan.
\newblock Topological aspects of poset spaces.
\newblock {\em Michigan Mathematical Journal}, 59(1):3--24, 2010.

\bibitem{sanders2025secondCountable}
Sam Sanders.
\newblock Second-countable spaces and reverse mathematics.
\newblock {\em Doc. Math.}, 2025.

\bibitem{SOSOA}
Stephen~G. Simpson.
\newblock {\em Subsystems of second order arithmetic}.
\newblock Perspectives in Mathematical Logic. Springer-Verlag, Berlin, 1999.

\end{thebibliography}
\end{document}